\documentclass[reqno,a4paper]{amsart}
\usepackage[english,activeacute]{babel}
\usepackage{amssymb,amsmath,amsthm,amsfonts,mathrsfs,hyperref,mathtools}
\usepackage{dsfont}
\usepackage[cal=boondox]{mathalfa}
\usepackage[font=footnotesize]{caption}
\usepackage{tikz}
\usetikzlibrary{intersections,shapes,trees,arrows,spy,positioning,decorations,arrows.meta,decorations.pathreplacing,patterns,calc}
\usepackage{tikz-cd}
\usepackage{tkz-euclide}
\tikzset
{marking1/.style=
	{decoration=
		{markings,
			mark= between positions 0.03 and 1 step 2 mm with {\arrow[line width=0.5pt]{>}}
		},
		postaction=decorate
	}
}
\tikzset
{marking2/.style=
	{decoration=
		{markings,
			mark= between positions 0.03 and 1 step 2 mm with {\arrow[line width=0.5pt]{<}}
		},
		postaction=decorate
	}
}

\usepackage{tikzscale}

\definecolor{darkgreen}{rgb}{.125,.5,.25}
\definecolor{darklavender}{rgb}{0.5, 0, 0.5}
\definecolor{navyblue}{rgb}{0.0,0.14,0.4}

\usepackage{pgfplots}

\usepackage{xcolor}
\usepackage[numbers,square,sort]{natbib}
\usepackage{tabularx}
\usepackage{enumitem}  
\usepackage[T1]{fontenc}
\usepackage{url}  

\usepackage{mdframed}

\theoremstyle{plain}
\newtheorem{theorem}{Theorem}[section]
\newtheorem{lem}[theorem]{Lemma}
\newtheorem{lemma}[theorem]{Lemma}
\newtheorem{conj}{Conjecture}
\newtheorem{proposition}[theorem]{Proposition}

\theoremstyle{definition}

\newtheorem{definition}[theorem]{Definition}
\newtheorem{example}[theorem]{Example}

\theoremstyle{remark}
\newtheorem{remark}[theorem]{Remark}
\newtheorem{assu}[theorem]{Assumption}

\numberwithin{equation}{section}

\newcommand{\Eb}  {{\mathbb E}}
\newcommand{\Nb}  {{\mathbb N}}
\newcommand{\Rb}  {{\mathbb R}}

\newcommand{\Pb}  {{\mathbb P}}

\newcommand{\As} {{\mathcal A}}

\newcommand{\Fs} {{\mathcal F}}

\newcommand{\Hs} {{\mathcal H}}
\newcommand{\Is} {{\mathcal I}}

\newcommand{\Ms} {{\mathcal M}}

\newcommand{\bell} {{\boldsymbol \ell}}
\newcommand{\bW} {{\boldsymbol W}}
\newcommand{\bB} {{\boldsymbol B}}
\newcommand{\bZ} {{\boldsymbol Z}}

\newcommand{\bb} {{\boldsymbol \beta}}
\newcommand{\bl} {{\boldsymbol \lambda}}
\newcommand{\bc} {{\boldsymbol c}}
\newcommand{\omu} {{\mu_{\otimes}^K}}
\newcommand{\bmu} {{\boldsymbol \mu}}
\newcommand{\bxi} {{\boldsymbol \xi}}
\newcommand{\bpi} {{\boldsymbol \pi}}
\newcommand{\bp} {{\boldsymbol p}}
\newcommand{\bPi} {{\boldsymbol \Pi}}
\newcommand{\bbf} {{\boldsymbol f}}

\newcommand{\be} {{\boldsymbol e}}
\newcommand{\bx} {{\boldsymbol x}}

\newcommand{\bv} {{\boldsymbol v}}
\newcommand{\bro} {{\boldsymbol \rho}}
\newcommand{\bu} {{\boldsymbol u}}
\newcommand{\bn} {{\boldsymbol n}}
\newcommand{\bk} {{\boldsymbol k}}

\newcommand{\dd} {{\rm d}}

\newcommand{\amax}{\alpha_{\mathrm{max}}}
\newcommand{\amin}{\alpha_{\mathrm{min}}}
\newcommand{\wmax}{w_{\mathrm{max}}}
\newcommand{\lmax}{\lambda_{\mathrm{max}}}

\usepackage[foot]{amsaddr}

\title{Structured coalescents, coagulation equations and multi-type branching processes}

    \author{Fernando Cordero$^{(1)}$}
    \address[1]{BOKU University, Institute of Mathematics, Department of Natural Sciences and Sustainable Ressources, Gregor-Mendel-Strasse 33/II, 1180 Vienna, Austria}
    \email{fernando.cordero@boku.ac.at}
	
	\author{Sophia-Marie Mellis$^{(2)}$}
    \address[2]{ Bielefeld University, Faculty of Technology, Box 100131, 33501 Bielefeld, Germany}
    \email{smellis@techfak.uni-bielefeld.de}

    \author{Emmanuel Schertzer$^{(3)}$}
    \address[3]{University of Vienna, Department of Mathematics, Oskar-Morgenstern-Platz 1, 1090 Vienna, Austria}
    \email{emmanuel.schertzer@univie.ac.at}
    
\date{\today}%

\begin{document}

	\begin{abstract}
    Consider a structured population consisting of $d$ colonies, with migration rates proportional to a positive parameter $K$. We sample $N_K$ individuals, distributed evenly across the $d$ colonies, and trace their ancestral lineages backward in time. Within each colony, we assume that any pair of ancestral lineages coalesces at a constant rate, as in Kingman’s coalescent. We identify each ancestral lineage with the set, or block, of its sampled descendants, and we encode the state of the system using a $d$-dimensional vector of empirical measures; the $i$-th component records the blocks present in colony $i$ together with the initial locations of the lineages composing each block.

    We are interested in the asymptotic behavior of the process of empirical measures such as $K \to \infty$. We consider two regimes: the critical sampling regime, where $N_K \sim K$, and the large-sample regime, where $N_K \gg K$. After an appropriate time rescaling, we show that the process of empirical measures converges to the solution of a $d$-dimensional coagulation equation. In the critical sampling regime, the solution can be represented in terms of a multi-type branching process. In the large-sample regime, the solution can be represented in terms of the entrance law of a multi-type continuous-state branching process.

	\end{abstract}	

    \maketitle

    \noindent{\slshape\bfseries MSC 2020.} Primary: 60J90; 60G57 Secondary: 60J95; 60B10; 60J80; 35Q92 \\
    
    \noindent{\slshape\bfseries Keywords. }{Structured coalescent; coagulation equations; multi-type branching processes; multi-type Feller-diffusions}
    

\section{Introduction}
Coalescents, branching processes, and coagulation equations represent three fundamental approaches to modeling the dynamics of interacting particle systems. Each captures a distinct yet interconnected facet of stochastic evolution: coalescents describe the merging of ancestral lineages backward in time \cite{kingman1982coalescent, Ki82b, MM98, Pit99, Sa99}; branching processes model population growth and reproduction mechanisms forward in time \cite{harris63, athreya2004branching, jagers89, k97, Lamp67, FouMa19}; and coagulation equations provide deterministic approximations for the evolution of cluster sizes in systems undergoing mass-conserving mergers \cite{Kokh88, Dub94, Dr97, Norris99}. \\

The relationships between these objects are theoretically rich and practically significant. For example, it is well known \cite{aldous1999deterministic, deaconu2000, bertoin2006} that the mean-field limit of the properly scaled evolution of block sizes in Kingman’s coalescent is given by a Smoluchowski coagulation equation. Furthermore, probabilistic representations of one-dimensional coagulation equations in terms of single-type branching processes can be found, for instance, in \cite{deaconu2000}. In the multi-type setting, these connections are more involved. For example, recent works \cite{JoAm23, GKNP} show that the genealogies of multi-type continuous-state branching processes can be described (at least locally) by exchangeable multi-type coalescents \cite{johnston2023multitype, mohle2024multi, FM}. Along similar lines, one of the present authors (E.S.) introduced a nested coalescent model \cite{lambert2020coagulation, blancas2019nested}, in which gene lineages evolve within a larger species tree. That work established a connection between the nested coalescent and a transport–coagulation equation, and demonstrated that the corresponding deterministic PDE admits a stochastic representation in terms of a branching CSBP.  

In this work, we establish a related connection between a structured Kingman coalescent \cite{notohara1990coalescent}, a multidimensional Smoluchowski coagulation equation, and multi-type branching processes. Our results generalize some of the aforementioned results available for the Kingman coalescent and the nested coalescent to a multi-type (or structured) setting, and allow for sample-size scaling regimes that differ from those studied in the classical literature. \\

Our starting point is the classical \emph{structured coalescent} of \cite{notohara1990coalescent} where a population is structured into $d$ colonies connected through migration.  We sample $N_K$ individuals, distributed across the $d$ colonies, and trace their genealogies backwards in time. Within a colony, pairs of ancestral lineages are assumed to coalesce as in a Kingman coalescent.  \\

 We will examine the regime of fast migration, where migration occurs at a much faster rate than coalescence. Specifically, we consider migration rates of order $\mathcal{O}(K)$ with $K \to \infty$. This regime has been widely investigated in the literature in the context of population genetics (see, e.g., \cite{charlesworth2009effective} for a review). For a {\it fixed sample size}, Norborg and Krone demonstrated in \cite{nordborg1997structured,nordborg2002separation} that the genealogical structure  effectively collapses in the fast migration limit, resulting in an averaged behavior across colonies. In this extreme case of complete structural collapse, the process converges to the standard one-dimensional Kingman coalescent, which corresponds to a fully mixed population.
  \\

A key assumption in the previous averaging result is that the sample size remains fixed as the migration scale $K$ increases. In contrast, this article focuses on a different asymptotic regime, where the sample size $N_K$ grows with the migration scale, i.e., $N_K \to \infty$ as $K \to \infty$. 
More precisely, we investigate two distinct sampling regimes: the critical sampling regime, where $N_K \sim K$, and the large sampling regime, where $N_K \gg K$. In both cases, each ancestral lineage retains information about the block of sampled individuals it traces back to and the aforementioned averaging does not hold anymore. To capture this result, we represent the system's state as a $d$-dimensional vector of empirical measures, where the $i$-th component captures the blocks present in colony $i$ and their corresponding block configurations -- the initial locations of the lineages within each block.
 We will show that this process converges to the solution of a multi-dimensional Smoluchowski-type coagulation equation under an appropriate small time scaling. Further, we will show that the resulting equation admits a natural probabilistic interpretation in terms of multi-type branching processes. Taken together, our approach hints at a unified framework that connects structured coalescents, coagulation equations, and multi-type branching processes.  \\

A natural question arising from this work is how to characterize the Site Frequency Spectrum (SFS) of the structured coalescent in the regime of fast migration and large sample sizes. Assuming that neutral mutations occur at a constant rate along the branches of the coalescent tree, the $i^{\text{th}}$ component of the SFS (for $i \in [N_K-1]$) represents the number of mutations present in exactly $i$ leaves. This question has been studied in the context of $\Lambda$- and $\Xi$-coalescents (see, e.g., \cite{berestycki2007beta,gonzalez2024asymptotics}), where it has been shown that, in the limit of large sample size, the lower end of the SFS is influenced by the small-time behavior of the coalescent process. We believe that the small-time asymptotics derived in our work may yield analogous results in the structured coalescent setting. Investigating this will be the focus of future research.
\\

The remainder of the paper is organized as follows. Section \ref{secmodel} introduces the model and states the main results. Subsequent sections are devoted to proving these results. Section \ref{seccongen} proves the convergence of the generators of the empirical measure process and states a comparison result that allows us to bound the number of blocks in our structured coalescent between the number of blocks of in two Kingman coalescents with different coalescent rates. Section \ref{secgentight} establishes the tightness of the sequence of empirical measures and characterizes their accumulation points. We conclude the proofs of our main results in Section \ref{secconvergenceproofs}. 
	

\section{Model and main results}\label{secmodel}
In this section we formalize the definition of the structured coalescent that will be the object of our analysis and we state our main results. 
\subsection{The model}
We consider a structured coalescent with $d$ colonies, where $d \in \mathbb{N}$ is fixed throughout the manuscript. The process evolves as follows. Within each colony, blocks coalesce as in Kingman’s coalescent (i.e. at a constant rate per pair), and blocks migrate between colonies at rates proportional to a scaling parameter $K$. As shown, for example, in \cite{nordborg1997structured}, such models arise naturally as genealogies of population-level systems. Although we use the terminology of “colonies”, the same model also describes a multi-type population, with types playing the role of colonies; migration then corresponds to mutation, and coalescence occurs within types exactly as within colonies. The aim of this article is to analyze how the sample size $N_K$ affects the corresponding ancestral structures at small times as $K \to \infty$.
\\

    Let us now formalize the previous description. Each individual in the population is identified by a unique number in $[N_K]\coloneqq\{1,\ldots,N_K\}$; each colony is identified by a unique number in $[d]$, which we will often refer to as a color. The state of the system is encoded by a colored partition as defined below.

	\begin{definition}[Colored partition]
Let $N \in \mathbb{N}$. We refer to sets of the form 
$S = \{(1,c_1), \ldots, (N,c_N)\}$, $c_1,\ldots,c_N \in [d]$,
as a \emph{coloring} of $[N]$. We call $\boldsymbol{\pi} \coloneqq (\pi_i)_{i \in [d]}$ a \emph{colored partition} of a coloring $S$ of $[N]$ if $\pi_1,\ldots,\pi_d$ are disjoint (possibly empty) collections of non-empty subsets of $S$ and if $\pi_1 \cup \pi_2 \cup \cdots \cup \pi_d$ is a partition of $S$; the elements of $\pi_i$ are called \emph{blocks of color $i$}. If $S$ is a coloring of $[N]$, we denote by $\mathcal{S}_S^N$ the set of all colored partitions of $S$, and by $\mathcal{S}^N$ the set of all colored partitions of some coloring of $[N]$. We equip both $\mathcal{S}_S^N$ and $\mathcal{S}^N$ with the discrete topology.
\end{definition}
\begin{example}The set
    \label{ex:coloring}
    $$S= \{ {\color{navyblue} 1},
    {\color{darkgreen!60}
    2},{\color{navyblue}
    3, 4},{\color{darkgreen!60} 5} \} $$ 
    is a coloring of $[5]$ with $d=2$ colors, with the coloring represented directly in scriptcolor rather than by adding an additional coordinate encoding the color; blue represents colour~1 and green represents colour~2.
    $$
    {\bpi} =\{ \underbrace{{\color{navyblue} \{} {\color{navyblue} 1},{\color{darkgreen!60} 5} {\color{navyblue} \}},{\color{navyblue} \{}{\color{navyblue}
    3}{\color{navyblue}\}},}_{\pi_1}
  \underbrace{  {\color{darkgreen!60} \{}{\color{darkgreen!60}
    2},{\color{navyblue}
    4}{\color{darkgreen!60} \}}}_{\pi_2} \}
    $$
    is a colored partition of $S$ with $2$ blue blocks and 1 green block. Note that the coloring operates on two levels: the individual elements within each block are colored, and the blocks themselves receive their own colors.
    \end{example}
We now describe the structured coalescent as a Markov process on the space of colored partitions. Fix $K \geq 1$ and $N_K \in \mathbb{N}$. Let $\boldsymbol{W} \coloneqq (w_{i,j})_{i,j \in [d]}$ be a primitive matrix (that is, $\boldsymbol{W}$ has nonnegative entries and there exists $n \in \mathbb{N}$ such that $\boldsymbol{W}^n$ has all entries strictly positive), and let $\boldsymbol{\alpha} \coloneqq (\alpha_i)_{i \in [d]} \in \mathbb{R}_+^d$. The structured coalescent $\boldsymbol{\Pi}^K \coloneqq (\boldsymbol{\Pi}^K(t))_{t \geq 0}$ is the continuous-time Markov chain on $\mathcal{S}^{N_K}$ defined by the following dynamics:

\begin{enumerate}
\item At time $0$, start with a coloring of $[N_K]$. The partition $\Pi_0$ is the set of singletons whose coloring coincides with the coloring of $[N_K]$. For instance, for the coloring of $S$ in Example~\ref{ex:coloring},
$$
\Pi^K_0= \{ {\color{navyblue} \{}{\color{navyblue} 1}{\color{navyblue} \}},
    {\color{darkgreen!60} \{}{\color{darkgreen!60}
    2}{\color{darkgreen!60} \}},
    {\color{navyblue} \{}{\color{navyblue} 3} {\color{navyblue} \}},
    {\color{navyblue} \{}{\color{navyblue} 4}{\color{navyblue} \}},
    {\color{darkgreen!60} \{}{\color{darkgreen!60} 5}{\color{darkgreen!60} \}} \}.
    $$
\item Each block changes its color from $i$ to $j$ at rate $K w_{i,j}$. For instance
$$
 {\color{navyblue} \{} {\color{navyblue} 1},{\color{darkgreen!60} 5} {\color{navyblue} \}} \to {\color{darkgreen!60} \{} {\color{navyblue} 1},{\color{darkgreen!60} 5} {\color{darkgreen!60} \}} \ \ \mbox{at rate $Kw_{1,2}$.}
$$
\item  For every $i$ in $[d]$,  each pair of blocks of color $i$ coalesce into a single block of color $i$ at rate $\alpha_i$. For instance
$$
{\color{navyblue} \{} {\color{navyblue} 1},{\color{darkgreen!60} 5} {\color{navyblue} \}},{\color{navyblue} \{}{\color{navyblue}
    3}{\color{navyblue}\}} \ \to \ {\color{navyblue} \{} {\color{navyblue} 1},{\color{darkgreen!60} 5}, {\color{navyblue}
    3}{\color{navyblue}\}}  \ \ \mbox{at rate $\alpha_1$.}
$$
Pairs of different colors do not coalesce.
 \end{enumerate}
If we envision the coalescent as a random ultrametric tree, the color of a block and its internal coloring can be interpreted as follows. The color of a block at time $t$ is the position (colony) of the corresponding lineage. The internal coloring records the labels and colors of the leaves supported by this lineage at time $t$ in the past. 
\begin{assu}\label{assu1}
Let $L_i^K(t)$ be the number of blocks of color $i$ at time $t$.
  There is a vector $\bb\coloneq(\beta_i)_{i \in [d]}\in\Rb_+^d$, such that, for each $i\in[d]$,   
     $$\frac{L_i^K(0)}{N_K} \xrightarrow[K\to\infty]{}\beta_i.$$
     
\end{assu} 

Since the matrix $\bW$ is primitive, it has a unique stationary distribution, denoted by $\bxi \coloneqq (\xi_i)_{i \in [d]}$, characterized by
\[
\sum_{j \in [d]\setminus \{i\}} \xi_j w_{j,i}
= \xi_i \sum_{j \in [d]\setminus \{i\}} w_{i,j}.
\]
\subsection{The empirical measure}
In this work, we focus on the asymptotic properties of a functional of the coalescent $\boldsymbol{\Pi}^K$, which encodes both the colony sizes and the \emph{color configurations} within blocks (see Fig.~\ref{figmodel}). To formalize this, we introduce a few definitions. We say that the (color) configuration of a block is $\boldsymbol{k} = (k_i)_{i \in [d]} \in [N_K]_0^d$ if and only if the block contains $k_1$ elements of color $1$, $k_2$ elements of color $2$, and so on. For instance, the configuration of the blue block
$
{\color{navyblue} \{} {\color{navyblue} 1},{\color{darkgreen!60} 3}, {\color{darkgreen!60} 5} {\color{navyblue} \}}$
is $\textcolor{navyblue}{(1}, \textcolor{darkgreen!60}{2}\textcolor{navyblue}{)}$. \\

\begin{figure}[b!]
 \scalebox{0.7}{
		\includegraphics[width=0.7\textwidth]{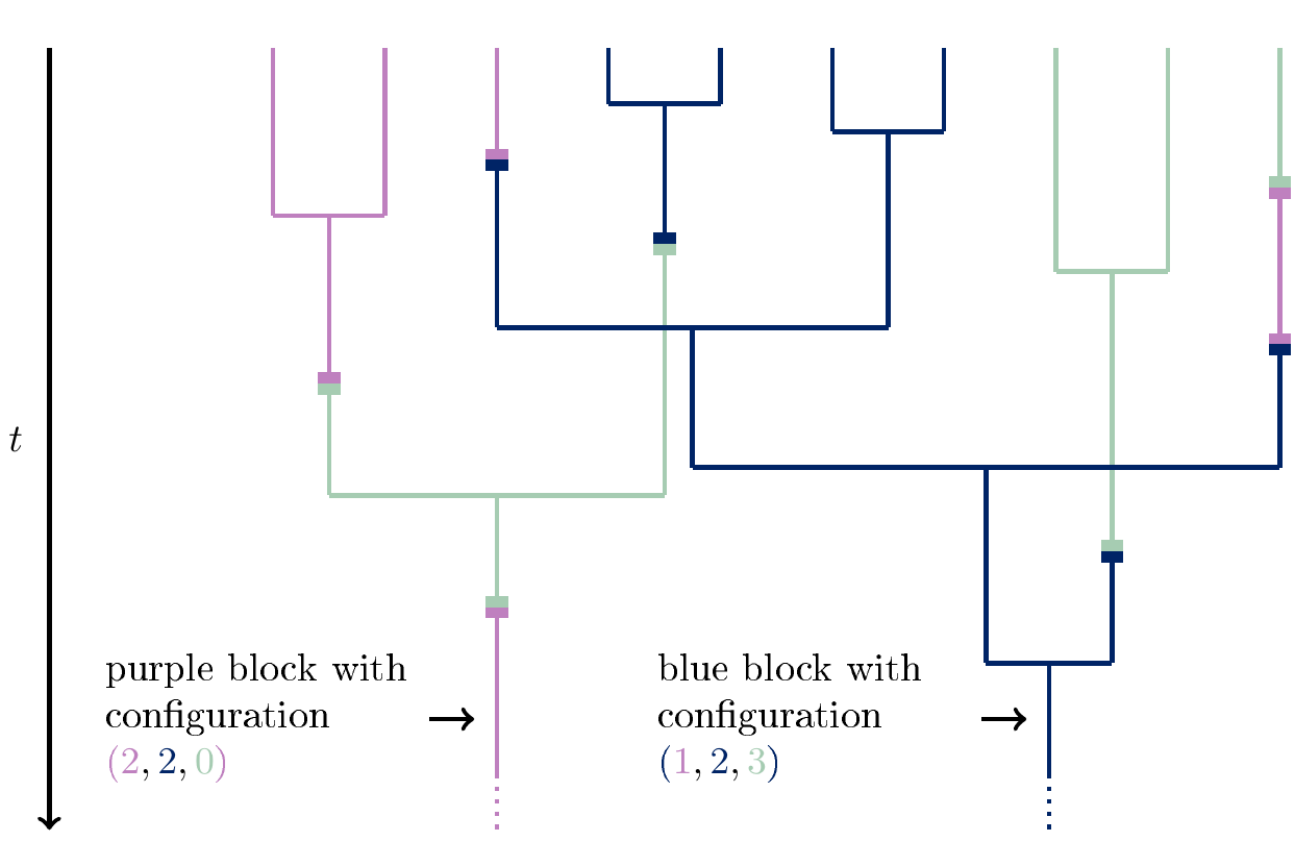}
        }
		\caption{An illustration of the structured coalescent for $d=3$, $N_{K}=10$. Blocks are classified according to their colors, which code for the different \emph{colonies}. A change in color represents a migration event. Specifically, multi-colored squares indicate a migration from colony $i$ to colony $j$, where the upper color corresponds to the origin colony $i$, and the lower color to the destination colony $j$.}\label{figmodel}
\end{figure}

For every time $t\in\mathbb{R}_+,i\in[d]$,
we define the empirical measure $\nu_{i}^{K}(t,\cdot)\in {\mathcal M}_{f}([N_K]^{d}_0)$ via
$$
 \ \nu_i(t,
\{\bk\}) \  \coloneqq \ \#\{ \mbox{blocks of color $i$ with configuration ${\bk}$} \},\quad {\bk} = (k_i)_{i \in [d]}\in [N_K]_0^{d}. \ $$

In words, $\nu_i(t,\{\boldsymbol{k}\})$ is the number of lineages at time $t$ located in colony $i$ that carry $k_1$ leaves of color $1$, $k_2$ leaves of color $2$, and so on. As anticipated at the beginning of this section, our aim is to understand the interplay between the sample size $N_K$ and the scale $K$ at which migrations (changes of color) occur in the model. We therefore distinguish between two regimes, depending on the asymptotic behaviour of $\gamma_K \coloneqq N_K / K$. More precisely, we consider
		\begin{enumerate}
			\item \emph{The critical sampling regime}:
			\[ \gamma_K \underset{K \to \infty}{\longrightarrow} c>0.\]
			\item \emph{The large sampling regime}:
			\[ \gamma_K \underset{K \to \infty}{\longrightarrow} \infty.  \]
		\end{enumerate}

       In the critical sampling regime, migration and coalescence occur on the same time scale as long as the total number of blocks remains of order $\mathcal{O}(K)$. In the large-sampling regime, coalescence dominates as long as the number of blocks exceeds order $K$, and the two mechanisms act on a common time scale only once the number of blocks has fallen to order $K$. We therefore rescale time by a factor $1/K$. This rescaling allows us to focus on the phase in which coalescence and migration operate on comparable time scales, after a very brief initial transient during which coalescence prevails (see Fig.~\ref{figtimescale} for an illustration of the time scaling). As a consequence, in the large-sampling regime, blocks quickly become very large, and we therefore need to introduce an appropriate scaling of the blocks. More precisely, and to avoid unnecessary case distinctions when scaling blocks, we introduce the scaling parameter
 \begin{equation}
    s_{K} \coloneq \begin{cases}
        1, &\quad\text{in the critical sampling} \\
        \gamma_{K}, &\quad\text{in the large sampling}. \label{eqdefsk}
    \end{cases}  
\end{equation} Note that in both regimes
\begin{equation}\label{govers}
b\coloneqq \sup_{K}\gamma_K/s_K<\infty.
\end{equation}               
        With this intuition in mind, we consider the rescaled process $(
        \mu^{K}_i(t,\dd x))_{t \geq 0}$ valued in ${\mathcal M}_{f}(\mathbb{R}_+^d)$ equipped with weak topology as
        $$
        \ \ \  \mu_{i}^K(t,\dd x) 
        \ \coloneqq  \frac{1}{K}\ \nu^K_i\left(\frac{t}{K}, s_K\, \dd x  \right),\qquad i\in[d],\  t\geq 0,$$ 
        where the measure $\nu_i^K\!\left(u, s_K\, \dd x\right)$ denotes the pushforward of $\nu_i^K(u,\dd x)$ under the map $x \mapsto s_K x$. We will often write $\boldsymbol{\mu}^K$ for the vector of empirical-measure processes $(\mu_i^K)_{i \in [d]}$. It is important to note that three distinct scalings are involved: time and the total mass of the measure are each scaled by $1/K$, while the block configurations are scaled by $s_K$.

	\begin{figure}[h]
    \scalebox{0.7}{
		\includegraphics[width=\textwidth]{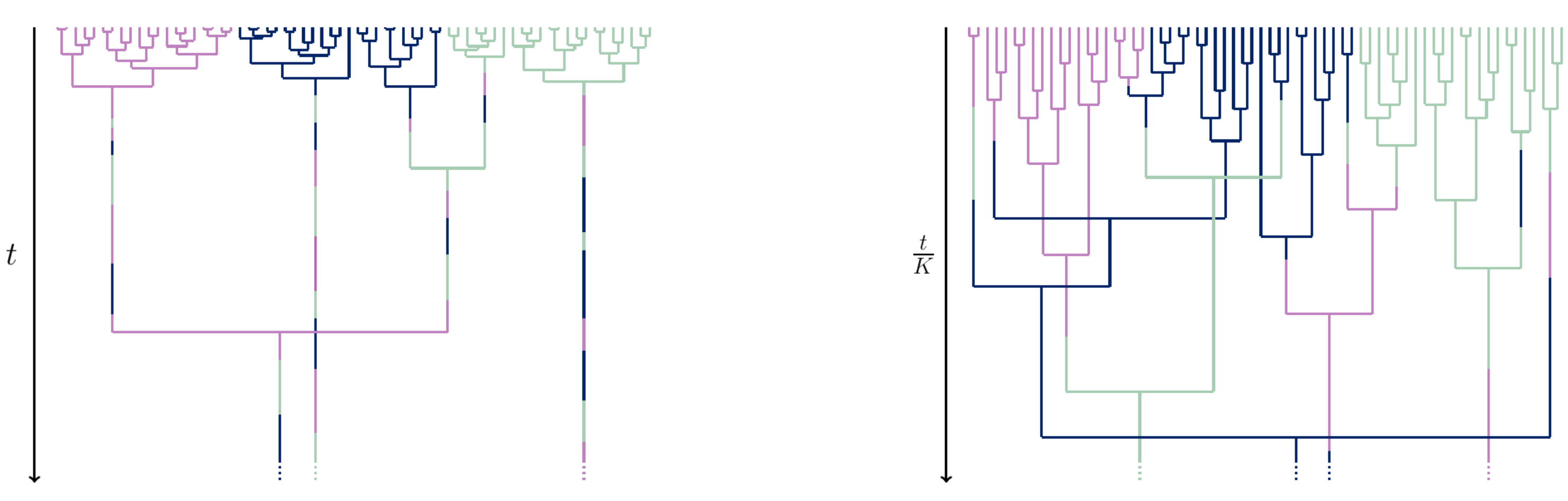}
        }
		\caption{ 
        An illustration of the effect of the time scaling $t\mapsto t/K$.}\label{figtimescale}
	\end{figure}
\subsection{The one-dimensional case (the critical sampling regime)}

To illustrate the type of asymptotic behavior one may expect for $(\boldsymbol{\mu}^{K}(t,\dd x))_{t \ge 0}$, we first recall known results in the critical sampling regime for $d=1$, before turning to our own contributions. Throughout the remainder of this section, we therefore assume that $d=1$ and $N_K = K$. \\

With only a single color, the coalescent no longer needs to be defined on colored partitions; in this case, $\boldsymbol{\Pi}^{K} = \Pi^{K}$ reduces to the classical Kingman coalescent with coalescence rate $\rho = \alpha_1$. Describing block configurations then amounts to specifying block sizes, and we write $\mu^K$ and $\nu^K$ for the analogues of $\boldsymbol{\mu}^K$ and $\boldsymbol{\nu}^K$ under this identification. In this setting, it follows from \cite{bertoin2006, Norris99} that the rescaled process $(\mu^{K}(t))_{t \ge 0}$ converges to the solution of the discrete coagulation equation
\begin{align*}
  \partial_t u(t,n) &= \rho \Big( \frac{1}{2} \, u \star u(t,n) - \langle u(t), 1 \rangle \, u(t,n)\Big), \\
  u(0,n) &= \delta_{1,n},
\end{align*}
for all $t \ge 0$ and $n \in \mathbb{N}$. Furthermore, it is well known (see, for example, \cite{deaconu2000, aldous1999deterministic}) that the solution of this equation admits a natural probabilistic interpretation in terms of a branching process. \\

To see this, define, for $\lambda \in [0,1)$,
\[
v(t,\lambda) := 1 - \langle u(t), 1 \rangle + \sum_{n \in \mathbb{N}} u(t,n) \lambda^n.
\]
The function $v(t,\cdot)$ is the probability generating function of the probability distribution
\[
\big(1 - \langle u(t), 1 \rangle,\ (u(t,n))_{n \in \mathbb{N}}\big)
\]
on $\mathbb{N}_0$, and clearly $v(0,\lambda) = \lambda$. A direct computation shows that, for all $t \ge 0$,
\begin{align*}
\partial_t v(t,\lambda) 
&= \rho \left( \frac{1}{2} v(t,\lambda)^2 + \frac{1}{2} - v(t,\lambda) \right)
= h\big(v(t,\lambda)\big) - \rho\, v(t,\lambda),
\end{align*}
where $h(x) = \frac{\rho}{2}(x^2 + 1)$ is the probability generating function of $\frac{\rho}{2}\delta_0 + \frac{\rho}{2}\delta_2$.

Since $v(0,\lambda) = \lambda$, this evolution coincides with the generating function of a continuous-time branching process in which each individual either gives birth to two offspring or dies, each at rate $\rho/2$, started from a single ancestor (see the backward Kolmogorov equation, e.g. \cite[Chap.~III.3, Eq.~(5)]{athreya2004branching}). For such a process, it is well known (see, e.g. \cite[Eq.~(17)]{kendall1948generalized}) that the one-dimensional marginals are geometric. Specifically, for all $n \in \mathbb{N}$ and $t \ge 0$,
\[
u(t,n) = q_t^2 (1-q_t)^{\,n-1},
\]
where $q_t = 1/(1 + \rho t/2)$ is the survival probability at time $t$, and $(q_t (1-q_t)^{n-1})_{n \in \mathbb{N}}$ is the distribution of the branching process at time $t$ conditioned on survival.

With this in mind, we are now prepared to state our main convergence results.
	\subsection{Convergence}
	In this section, we state the two main results concerning the convergence of the empirical measures $(\mu_i^{K})_{i \in [d]}$ as $K \to \infty$. Let us emphasize from the outset that the qualitative behaviour of the limit depends strongly on the parameter regime under consideration.

We start by introducing some notation. If $\eta_1, \eta_2 \in \mathcal{M}_f(\mathbb{R}^d)$, we denote by $\eta_1 \star \eta_2$ their convolution, that is,
\[
\eta_1 \star \eta_2(B) = \int \mathds{1}_B(\boldsymbol{x} + \boldsymbol{y}) \, \eta_1(\mathrm{d}\boldsymbol{x}) \, \eta_2(\mathrm{d}\boldsymbol{y}).
\]
Moreover, if $\eta \in \mathcal{M}_f(\mathbb{R}^d)$ and $f : \mathbb{R}^d \to \mathbb{R}$ is integrable with respect to $\eta$, we set
\[
\langle \eta, f \rangle \coloneqq \int f(\boldsymbol{x}) \, \eta(\mathrm{d}\boldsymbol{x}).
\]
If the measures involved have discrete support, these integrals are, as usual, understood as sums.

We now extend the convergence result from the one-dimensional critical sampling regime to the multidimensional setting.

\begin{theorem}[Critical sampling]\label{theocrit}
Assume that $\gamma_K \to c$ as $K \to \infty$. If Assumption~\ref{assu1} holds, then $(\mu_i^{K})_{i \in [d]}$ converges weakly, as $K \to \infty$, to the solution of the $d$-dimensional discrete coagulation equation
\begin{align}
\partial_t u_i(t,\boldsymbol{n}) 
&= \alpha_i \left( \frac{1}{2} u_i \star u_i(t,\boldsymbol{n}) - \langle u_i(t),1\rangle \, u_i(t,\boldsymbol{n}) \right)
+ \sum_{j \in [d]\setminus\{i\}} \big( w_{j,i} u_j(t,\boldsymbol{n}) - w_{i,j} u_i(t,\boldsymbol{n}) \big), \label{SE2} \\
u_i(0,\boldsymbol{n}) 
&= c \beta_i \, \delta_{\boldsymbol{e}_i,\boldsymbol{n}}, \label{ICa2}
\end{align}
for all $t \ge 0$, $\boldsymbol{n} \in \mathbb{N}_0^d \setminus \{\boldsymbol{0}\}$ and $i \in [d]$, where $(\boldsymbol{e}_i)_{i \in [d]}$ denotes the canonical basis of $\mathbb{R}^d$.
\end{theorem}

The dynamics of the coagulation equation are intuitive. When two blocks coalesce, the resulting block configuration is the sum of the configurations of the coalescing blocks, which explains the convolution term. Moreover, whenever a block with a given configuration merges with another block, that configuration is lost. The remaining terms account for migration: a block configuration is gained in colony $i$ when a block with that configuration migrates from another colony to $i$, and it is lost when a block migrates from colony $i$ to another colony. The initial condition reflects the fact that at time $t=0$ there are approximately $\beta_i K$ singletons in colony $i$. \\

The next theorem extends the convergence results to the large-sampling regime.

\begin{theorem}[Large sampling]\label{theolarge}
Assume that $\gamma_K \to \infty$ as $K \to \infty$. If Assumption~\ref{assu1} holds, then $(\mu_i^{K})_{i \in [d]}$ converges weakly, as $K \to \infty$, to the weak solution of the $d$-dimensional continuous coagulation equation
\begin{align}
\partial_t u_i 
&= \alpha_i \left( \frac{1}{2} u_i \star u_i - u_i \langle u_i,1\rangle \right)
+ \sum_{j \in [d]\setminus\{i\}} \big( w_{j,i} u_j - w_{i,j} u_i \big),
\quad i \in [d], \; t > 0, \label{SE1} \\
\lim_{t \to 0}& \int_{\mathbb{R}_+^d} \left(1 - e^{-\langle \boldsymbol{\lambda}, \boldsymbol{x} \rangle}\right) u_i(t,\mathrm{d}\boldsymbol{x})
= \lambda_i \beta_i, 
\quad \boldsymbol{\lambda} = (\lambda_i)_{i \in [d]} > \boldsymbol{0}, \; i \in [d]. \label{ICa1}
\end{align}
\end{theorem}

\begin{remark}
In the previous theorem, by a weak solution we mean that for $f \in C_b(\mathbb{R}_+^d)$,
\begin{align*}
\langle u_i(t), f \rangle
&= \langle u_i(0), f \rangle
+ \frac{\alpha_i}{2} \int_0^t \big( \langle (u_i \star u_i)(s), f \rangle 
- 2 \langle u_i(s),1\rangle \langle u_i(s), f \rangle \big)\, \mathrm{d}s \\
&\quad + \sum_{j \in [d]\setminus\{i\}} \int_0^t 
\big( w_{j,i} \langle u_j(s), f \rangle - w_{i,j} \langle u_i(s), f \rangle \big)\, \mathrm{d}s.
\end{align*}

In the discrete setting underlying the critical sampling regime, weak convergence is defined analogously, but with the full class of bounded functions on $[N]_0^d$ as test functions. Since this class contains indicator functions, the notions of weak and strong solutions coincide; for this reason, we do not distinguish between them in Theorem~\ref{theocrit}. Moreover, in both Theorem~\ref{theocrit} and Theorem~\ref{theolarge}, existence of (weak) solutions is part of the conclusion.
\end{remark}

The form of the initial condition in \eqref{ICa1} is well known in the context of continuous-state branching processes (see Section~\ref{secuniqueness}). In the present setting, it reflects the fact that we start with $N_K$ blocks while rescaling total mass only by $K$, which leads to a singularity at $t=0$ as $K \to \infty$. In this regime, coalescence dominates the dynamics until the number of blocks reaches order $K$. We therefore expect migration to have little effect near time zero, and hence that in colony $i$ one has $\langle \boldsymbol{\lambda}, \boldsymbol{x} \rangle \approx \lambda_i x_i$. This, in turn, suggests that as $t \to 0$,
\[
\int_{\mathbb{R}_+^d} \left(1 - e^{-\langle \boldsymbol{\lambda}, \boldsymbol{x} \rangle}\right) \mu_i^K(t,\mathrm{d}\boldsymbol{x})
\;\approx\; \lambda_i \int_{\mathbb{R}_+^d} x_i \, \mu_i^K(t,\mathrm{d}\boldsymbol{x})
\;\approx\; \lambda_i \frac{L_i^K(0)}{K s_K}
\;=\; \lambda_i \frac{L_i^K(0)}{N_K}.
\]
As $K \to \infty$, this expression converges to $\lambda_i \beta_i$, which is precisely the quantity appearing on the right-hand side of \eqref{ICa1}.

\subsection{Stochastic representation}\label{secuniqueness}
In this section, we generalize the stochastic representation established for the coagulation equation in the one-dimensional critical-sampling regime to the solutions of the coagulation equations obtained in the previous section. We derive these representations in the large-sampling regime; the critical-sampling regime is treated analogously in Appendix~\ref{secuniquenessd>1}. Throughout the remainder of this section, we therefore assume that $\gamma_K \to \infty$ as $K \to \infty$.\\

Assume that there exists a weak solution ${\bu}=(u_i)_{i \in [d]}$ of the continuum coagulation equation \eqref{SE1} under \eqref{ICa1}. Define
	$$
	v_i(t,\boldsymbol{\lambda}) \ \coloneqq  \frac{1}{
    \beta_i} \int_{\mathbb{R}_{+}^d} \left(1-e^{-\left<{\bl},{\bx}\right>}\right) u(t,\dd \bx).
	$$
    A direct computation shows that $\bv=(v_i)_{i \in [d]}$ solves the multi-dimensional ODE
	\begin{equation*}
	 \partial_t v_i(t,\boldsymbol{\lambda}) \ =  \ - \psi_i( {\bv}(t,\boldsymbol{\lambda}) ) \  \mbox{and $v_i(0,\boldsymbol{\lambda}) = \lambda_i,$} 
     \end{equation*}
  {where }  
  \begin{equation*}
	\psi_i({ \boldsymbol{\lambda}}) \ = \ \frac{1}{2} (\alpha_i \beta_i ) \lambda_i^2 -  \sum_{j \in [d] \setminus \{i\}} \bigg(w_{j,i} \frac{\beta_j}{\beta_i}\lambda_j - w_{i,j} \lambda_i\bigg).
	\end{equation*}
We identify $\bv$ as the Laplace exponent of the ${\boldsymbol{\psi}}$-CSBP, or equivalently 
the $d$-dimensional Feller diffusion 
     ${\bZ}$ solution of the SDE (see, for example, \cite[Thm.~3]{watanabe69})
\begin{equation}\label{eq:feller}
  \dd Z_i(t) \ = \ \sqrt{\alpha _i \beta_i Z_i(t)} \dd B_i(t) \ + \ \sum_{j \in [d] \setminus \{i\}} \Big(w_{j,i} \frac{\beta_j}{\beta_i} Z_{j}(t) - w_{i,j} Z_i(t)\Big) \dd t, \quad i\in[d],
	\end{equation} 
	where $\bB$ is a standard $d$-dimensional Brownian motion. Thus, for every $\bx=(x_i)_{i \in [d]}$, we have
\begin{equation}\label{eq:laplace}
	\mathbb{E}_{\bx}[ e^{-{\langle \boldsymbol{\lambda}}, \bZ(t) \rangle}]  =  e^{-\left<{ \bx}, \bv(t,{ \boldsymbol{\lambda}}) \right>}.
	\end{equation}
	Let $Q_i$ be the entrance law of the process starting from state $i$, i.e.
	\[ Q_i(t,A) \coloneqq \lim_{x\to 0+} \frac{ \mathbb{P}_{x \be_i}(Z(t)\in A) }{x},\quad t>0. \]
Using standard arguments (see, for example, \cite[Thm.~8.6]{Li2011}), one can deduce from (\ref{eq:laplace}) that for $t>0$
	\[ 	v_i(t,\boldsymbol{\lambda})  =  \int_{\mathbb{R}_{+}^d}  \left(1-e^{-\left<{\bl}, {\bx }\right>)} \right) Q_i(t,\mathrm{d}\bx). \]
    But recall that
    \[  v_i(t,\boldsymbol{\lambda}) = \frac{1}{\beta_i} \int_{\mathbb{R}_{+}^d} \left(1-e^{-\left<{\bl},{\bx}\right>}\right) u_i(t,\dd  \bx). \]
    We can then invert the Laplace transform and get the following result.
    \begin{theorem}\label{thm:stoch-represenatation-large} Assume that there exists a weak solution to
    coagulation equation \eqref{SE1} with initial condition \eqref{ICa1}. Then it admits the stochastic representation
    $$
    u_i(t,{\dd \bx}) 
    \ =  \ \beta_i Q_i(t,{\dd\bx}),\quad t>0,\, i\in[d],
    $$
    where $Q_i$ is the entrance law of ${\bZ}$ of the $d$-dimensional Feller diffusion (\ref{eq:feller})
    starting from state $i$. In particular, the solution to \eqref{SE1} with initial condition \eqref{ICa1} is unique.
    \end{theorem}
\begin{remark}
In the one-dimensional case, it is well known (see, for example, \cite[Chap.~4.2, Eq.~(4.6)]{pardoux16}) that the entrance law of the one-dimensional Feller diffusion is given by
\[
Q(t) = \frac{2}{t}\,\mathrm{Exp}\!\left(\frac{2}{t}\right),
\]
where $\mathrm{Exp}(\alpha)$ denotes the exponential distribution with parameter $\alpha > 0$. Consequently, in close analogy with the discrete case, one obtains an explicit solution to the continuous coagulation equation (see, for example, \cite{deaconu2000, aldous1999deterministic}) in the form
\[
u(t,x) = \frac{4}{t^2} \exp\Big(- \frac{2x}{t} \Big).
\]
Moreover, \cite[Sec.~3.1]{aldous1999deterministic} provides an alternative probabilistic interpretation of this result based on a spatial Poisson construction. For each $t>0$, let $\mathcal{P}(t)$ be a Poisson point process on $\mathbb{R}$ with intensity $r(t) \coloneqq 2/t$, coupled by independent thinning so that, for $0<s<t$, $\mathcal{P}(t)$ is obtained from $\mathcal{P}(s)$ by retaining each point with probability $r(t)/r(s)$. This yields a decreasing Markov family of Poisson configurations, defined via an entrance law since $r(t) \to \infty$ as $t \to 0$. The intervals between points form a renewal (Poisson-cluster-type) structure whose lengths represent cluster masses and merge upon point deletion. The induced interval-length distribution then evolves according to the continuous coagulation equation~\eqref{SE1} for $d=1$.
\end{remark}

  For the critical sampling, we will prove the following analogous result in Appendix~\ref{secuniquenessd>1}.
  \begin{theorem}\label{propuniquenesscrit}
 Assume that for each $i\in[d]$, we have $$ d_i\coloneqq \frac{c \alpha_i \beta_i}{2} - \sum_{j \in [d] \setminus \{i\}}\left( \frac{\beta_j}{\beta_i} w_{j,i}-w_{i,j}\right) >0. $$ 
Assume further that there exists a solution ${\bu}=(u_i)_{i \in [d]}$ to
    coagulation equation \eqref{SE2} with initial condition \eqref{ICa2}. Then $\bu$ admits the stochastic representation
	\[ u_i(t, \bn) \ = \ {c} \beta_{i} {\mathbb P}_{\be_i}\big( \bZ(t) \ = \ \bn\big),\quad t\geq0,\, \bn\in \mathbb{N}_{0}^{d}\setminus \{ \boldsymbol{0} \},\, i\in[d],\] where $\bZ(t) = (Z_i(t))_{ i\in[d]}$ denotes a continuous-time multi-type branching process, such that a particle of color $i\in[d]$
	\begin{itemize}
		\item branches at rate $\frac{{c}  \alpha_{i} \beta_{i}}{2}$,
	    \item dies at rate $d_i$,
	    \item makes a transition from $i$ to $j$ at rate $\frac{\beta_j}{\beta_i} w_{j,i}$.
	\end{itemize}	
\end{theorem}

\begin{remark}
It is important to note that, in the critical sampling regime, such a branching representation may fail when the coalescence rates are too small; it holds only for sufficiently large values of $c$, consistent with the large-sampling regime, where $c$ may be regarded as effectively infinite. In the large-sampling regime this phenomenon is particularly striking: regardless of the sampling procedure, the sampled genealogy is a continuous-state branching process whose branching mechanism depends both on the underlying coalescent and on the sampling itself. Although it is well known in the one-dimensional case that coalescents at small times are closely related to branching processes, it was far from clear that this correspondence would extend to higher dimensions.

Let us also observe that the stochastic representations introduced above imply uniqueness of solutions to our coagulation equations under their respective initial conditions. In the critical case, we provide in Appendix~\ref{secuniquenessd>1}, Lemma~\ref{udiscr}, an alternative proof of uniqueness that removes the assumption $d_i \ge 0$. These uniqueness results will play a crucial role in establishing the convergence of the empirical measures.
\end{remark}
\begin{remark}[Stochastic representation at equilibrium]
Note that if $\boldsymbol{\beta} = \boldsymbol{\xi}$ is the equilibrium probability measure, i.e.
\[
\sum_{j \in [d] \setminus \{i\}} \beta_j w_{j,i}
\;=\;
\beta_i \sum_{j \in [d] \setminus \{i\}} w_{i,j},
\qquad \, i \in [d],
\]
then, for all $i\in[d]$,
\[
d_i = \frac{c\,\alpha_i \beta_i}{2} > 0, 
\]
and hence the conclusion of the previous theorem holds automatically.

\end{remark}
\subsection{Conjectures on the site-frequency spectrum}
As discussed in the introduction, the original motivation for this work comes from the study of the site-frequency spectrum (SFS) of structured coalescents -- a central object in population genetics and a notoriously challenging problem -- in the regime of fast migration and large sample sizes. A classical approach in the literature assumes fast migration with a fixed sample size; in this setting, a slow-fast principle implies that the SFS becomes asymptotically indistinguishable from that of a one-dimensional Kingman coalescent, leading to dimension reduction and explicit formulas. While mathematically elegant, this regime is biologically unsatisfactory, since information about the underlying population structure is lost -- an effect often referred to as the collapse of structure (see, for instance \cite[Chap.~6.3]{etheridge12}). Preliminary calculations suggest that the situation changes markedly when the sample size is large and of the same order as, or larger than, the migration scale, which is precisely the regime studied in this paper. We conjecture that in this setting the SFS is directly related to the solution of the coagulation equations introduced above. While a detailed investigation is beyond the scope of the present work, we briefly outline our proposed approach below.

To begin, let $\tau_{K}$ denote the time to the \emph{most recent common ancestor} (MRCA)
\[ \tau_{K}  \coloneqq   \inf \Big\{t\geq0 : 
\sum_{i \in [d]} L^K_i(t) =1 \Big\}. \]
We now define the \emph{branch-length measure} $B^K$ as the random variable valued in
${\mathcal M}([N_K]^d_0)$ such that
\[ B^{K}(\{{\bk}\}) =  \sum_{i \in [d]}  \int_0^{\tau_{K}} \nu_i^K(s,\{{\bk}\}) \mathrm{d} s, \quad {\bk}\in [N_K]^d_0.
\] We also define the \emph{rescaled branch-length measure} as 
\begin{equation*}
{\mathcal B}^{K}({\mathrm{d} \bx }) =  B^K({{s_K}{\mathrm{d} \bx}})  =  \sum_{i \in [d]}  \int_0^{K \tau_{K}} \mu_i^K(s,{\mathrm{d} \bx}) \mathrm{d} s.
\end{equation*}

Then Theorem~\ref{theolarge} already yields
\[
\sum_{i \in [d]} \int_0^{A} \mu_i^K(s,\mathrm{d}\boldsymbol{x}) \, \mathrm{d}s
\;\Rightarrow\;
\sum_{i \in [d]} \int_0^{A} u_i(s,\mathrm{d}\boldsymbol{x}) \, \mathrm{d}s
\quad \text{as } K \to \infty,
\]
for any fixed time $A > 0$. However, to analyze the asymptotic behavior of $\mathcal{B}^K$, we must also control the contributions arising after time $A$. We expect these contributions, which are related to the proportion of small blocks, to become negligible as $A$ grows. Making this intuition rigorous, however, would require sharp bounds on the growth of block sizes in the structured coalescent, and establishing such estimates lies beyond the scope of the present work.\\

In ongoing work, we aim to show the following.

\begin{conj}
    \label{prop:branch-length}
Assume that $\gamma_{K} \to \infty$ as $K \to \infty$.
Let $T_0=\inf\{t>0 : {\bZ}(t) = {\bf 0} \}$, where $\bZ$ denotes the solution of the SDE \eqref{eq:feller}, and define $Q_{\beta}=\sum_{i\in [d]} \beta_i Q_i$. Then 
$$
{\mathcal B}^{K}({\mathrm{d} \bx}) \ \Rightarrow \ {\mathcal B}^{\infty}({\mathrm{d} \bx}) = \sum_{i \in [d]}\int_0^\infty u_i(s,{\mathrm{d} \bx}) \mathrm{d} s = \int_{0}^{T_0} Q_{{\bb}}(s, \mathrm{d} \bx) \mathrm{d} s \quad\text{as }K \to \infty.
$$
In other words, ${\mathcal B}^\infty$ is the \emph{potential measure} associated to $Q_{\boldsymbol{\beta}}$.
    \end{conj}

   We now assume that neutral mutations (i.e. mutations that do not influence the genealogical dynamics of the coalescent) occur along the branches of the coalescent tree at a constant rate $\theta > 0$ (see Figure~\ref{figsitefrequency}). These mutations are interpreted in the spirit of the \emph{infinite-sites model}, which assumes that each mutation occurs at a unique site on the genome that has never mutated before. For each $\boldsymbol{k} \in [N_K]_0^d$, a mutation is said to be of type $\boldsymbol{k}$ if it affects $k_i$ leaves of color $i$ for all $i \in [d]$. We define the \emph{site-frequency measure} $S^K$ on $[N_K]_0^d$ by
\[
S^{K}(\boldsymbol{k}) \coloneqq \#\{\text{mutations of type $\boldsymbol{k}$ occurring before time $\tau_K$}\}, 
\qquad \boldsymbol{k} \in [N_K]_0^d.
\]
Formally, $S^K$ is a \emph{Poisson point process} on $[N_K]_0^d$ with intensity measure $\theta\, B^K(\mathrm{d}\boldsymbol{x})$.
 \\

    \begin{figure}[h]\centering
		\scalebox{0.7}{ \includegraphics[width=0.7\textwidth]{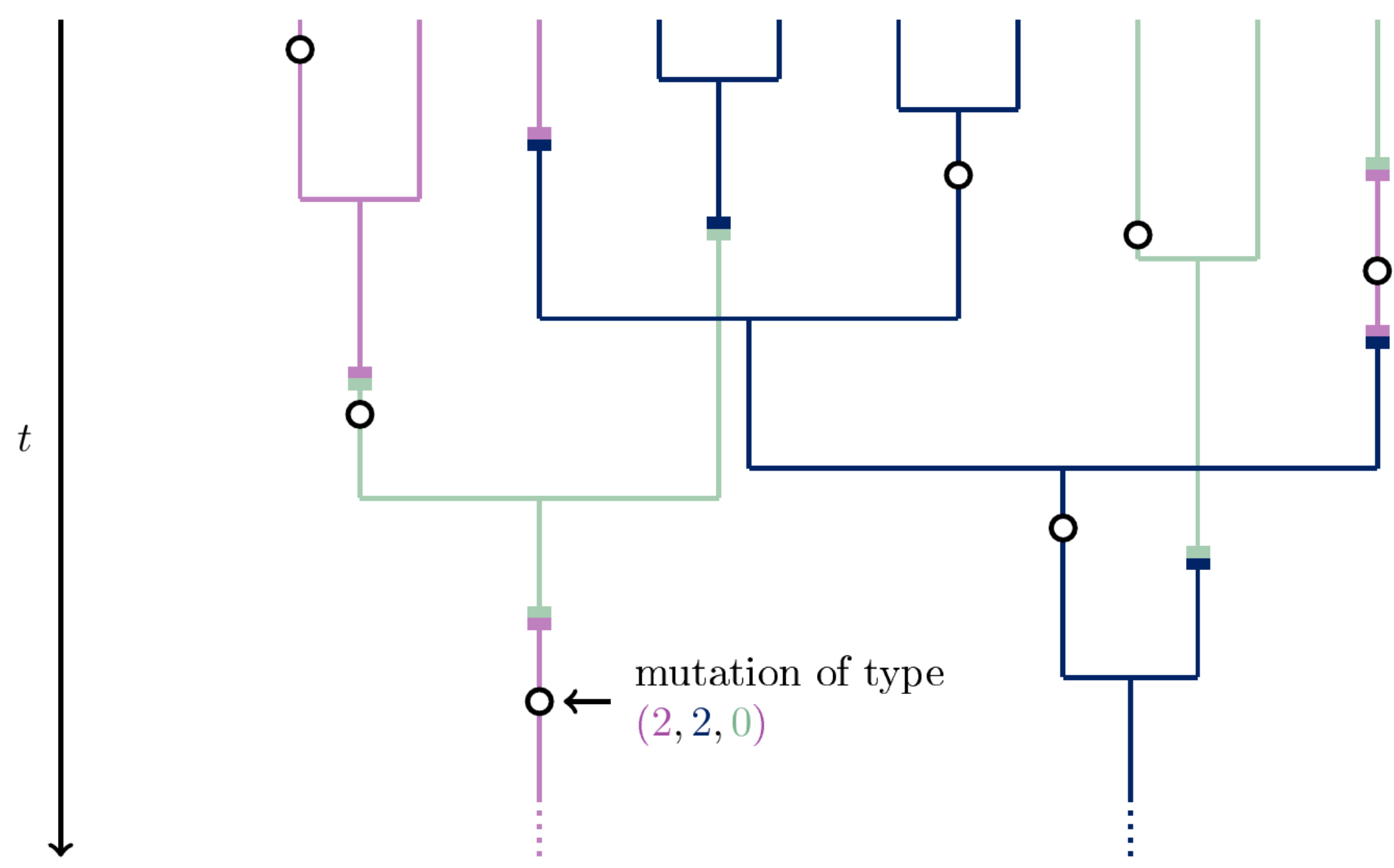} }
		\caption{The coalescent with mutations for $d=3$, $N_{K}=10$. Black circles indicate mutations occurring at constant rate $\theta$ per block.} \label{figsitefrequency}
	\end{figure}

As a direct consequence of Conjecture~\ref{prop:branch-length}, we expect the following result to hold.
\begin{conj}
Assume that $\gamma_{K} \to \infty$ as $K \to \infty$. Then
\[ S^{K}({\gamma_K}{\mathrm{d} \boldsymbol{x}}) \Rightarrow \mathrm{PPP}(\theta{\mathcal B}^{\infty}(\mathrm{d} \boldsymbol{x})) \quad\text{as } K \to \infty. \]
\end{conj}
In the case of critical sampling, we anticipate that an analogous result holds. \\

\subsection{Proof strategy for the main results}
We begin by recalling that the stochastic representation in Theorem~\ref{thm:stoch-represenatation-large} was already derived in Section~\ref{secuniqueness}. Its analogue in the critical regime, stated in Theorem~\ref{propuniquenesscrit}, will be proved in Appendix~\ref{secuniquenessd>1}. These results imply that the coagulation equations \eqref{SE2} and \eqref{SE1}, with initial conditions \eqref{ICa2} and \eqref{ICa1}, respectively, admit at most one (weak) solution; in the critical case, this conclusion requires the additional assumption that $d_i > 0$ for all $i \in [d]$. In Proposition~\ref{udiscr}, we show that uniqueness still holds in the critical case without this additional assumption. \\

Sections~\ref{seccongen}, \ref{secgentight}, and \ref{secconvergenceproofs} are devoted to the proofs of Theorems~\ref{theocrit} and~\ref{theolarge}, which rely on three main components: 
(1) convergence of the generators, 
(2) tightness, and 
(3) characterization of accumulation points. \\

Part~(1) is treated in Section~\ref{seccongen}, where we also derive moment estimates that are crucial for parts~(2) and~(3). Section~\ref{secgentight} addresses parts~(2) and~(3). Finally, in Section~\ref{secconvergenceproofs} we complete the proofs of Theorems~\ref{theocrit} and~\ref{theolarge}. We emphasize that our proof of existence of solutions to the coagulation equations, with their respective initial conditions, is probabilistic: it follows from parts~(2) and~(3) together with a careful analysis of the asymptotic initial conditions. This analysis is straightforward in the critical regime, but more delicate in the large-sampling regime. The uniqueness results discussed above play a key role in concluding the proofs. \\

Let us briefly comment on the differences in the analysis between our structured setting and the classical Kingman case. A key insight underlying several arguments in parts~(2) and~(3) is a coupling that allows us to transfer moment bounds from the classical (non-structured) Kingman coalescent. Beyond this, the structured setting renders parts~(1),~(2), and~(3) technically more involved than in the non-structured case, but without relying on fundamentally new concepts. Nonetheless, the analysis is far from a routine extension: a genuinely new difficulty arises in the study of the initial conditions in the large-sampling regime, where the main challenge is intrinsic to the structured nature of the coalescent.

\section{Generator convergence and moment bounds} \label{seccongen}
As announced in the introduction, the proofs of Theorems \ref{theolarge} and \ref{theocrit} rely on three classical ingredients:  
(1) convergence of the infinitesimal generators $A^K$ of the vector of empirical-measure processes 
$\bmu^K $, in an appropriate sense;  
(2) tightness of the sequence of empirical measures; and  
(3) identification of its accumulation points.  

In this section we address (1) (see Proposition \ref{propconvergencegenerators}) and establish estimates that will be useful when tackling (2) in the next section. Specifically, we derive $L^1$- and $L^2$-bounds for the sum of squared block sizes (see Lemma \ref{lemlemsecondmoment}), and we show that the block-counting process in our model can be sandwiched between the block-counting processes of two Kingman coalescents with different coalescence rates (see Lemma \ref{lemdelayedkingman}). This coupling will allow us to carry over moment bounds from the Kingman case, which we state at the end of the section (see Lemma \ref{moki}).

\subsection{The action of the generator}\label{subsecgen}
Let $\Hs$ be the set of real-valued functions on $(\mathcal{M}_f(\mathbb{R}_{+}^{d}))^{d}$ of the form 
	\begin{align}
		\bp \coloneqq (p_{i})_{i \in [d]} \in(\mathcal{M}_f(\mathbb{R}_{+}^{d}))^{d}
        \mapsto H^{F,\bbf}(\bp) \coloneqq  F (\langle \bp,\bbf\rangle )
        &\text{,} \label{eqH}
	\end{align} for some $F \in C(\Rb^d)$ and $\bbf\coloneqq(f_{i})_{i\in[d]}$ with  $f_{i} \in C(\mathbb{R}_{+}^d )$ for $i \in [d]$, where \[\langle \bp, \bbf \rangle \coloneqq( \langle p_{i}, f_{i} \rangle)_{i \in [d]}. \]
\begin{remark}
In the critical sampling regime, the space $\mathbb{R}_+^d$ must be replaced by $\mathbb{N}_0^d$. In particular, the test functions $f_i$ then need only be bounded and measurable (no continuity assumptions are required on a discrete space), whereas the regularity assumptions on $F$ used for generator calculations apply in both regimes. 

Throughout what follows (in particular, also in Section~\ref{secgentight}), definitions and arguments are presented in the large-sampling setting. The corresponding statements in the critical sampling regime are obtained by making the above distinction; with this modification, all definitions and arguments carry over.
\end{remark}

Recall that the \emph{configuration} of a block refers to the $d$-dimensional vector whose $i$-th coordinate, $i\in[d]$, records the number of atoms of color $i$. To describe the action of the generator $A^K$ of the empirical-measure process $\bmu^K$ on functions in $\mathcal{H}_s$, we decompose \[
A^{K} H^{F,\bbf}(\bp)
= A^{K}_{M} H^{F,\bbf}(\bp)
+ A^{K}_{C} H^{F,\bbf}(\bp),
\]
where $A^{K}_{M}$ and $A^{K}_{C}$ account for the contributions of migrations and coalescences, respectively. For the migration part, remember that any block configuration $\bc$ present in colony $i$ migrates to colony $j$ at rate $w_{i,j}K$. This adds to the $j$-th coordinate of $\bmu^K$ a mass $1/K$ at $\bc/s_K$ and removes the same mass from the $i$-th coordinate (see Fig.~\ref{figtransitions}, right). Since time is rescaled by $1/K$, we obtain

	\begin{align}\label{miggen}
		A^{K}_{M} H^{F,\bbf} (\bp)=K \sum_{i, j\in[d]}    w_{i,j}  \sum_{\bc \in \mathbb{N}_{0}^{d}}  p_{i}\Big(\Big\{\frac{\bc}{s_K}\Big\}\Big) \bigg[ F \bigg( \langle \bp, \bbf \rangle+\frac{\Delta_{i,j}\bbf(\frac{\bc}{s_K})}{K}   \bigg)  - F \left( \langle \bp, \bbf \rangle \right)  \bigg],
	\end{align} where $\Delta_{i,j}\bbf= \be_j f_{j}-\be_if_{i}$. 
	\begin{figure}[h!]
    \scalebox{0.8}{
		\includegraphics[width=0.6\textwidth]{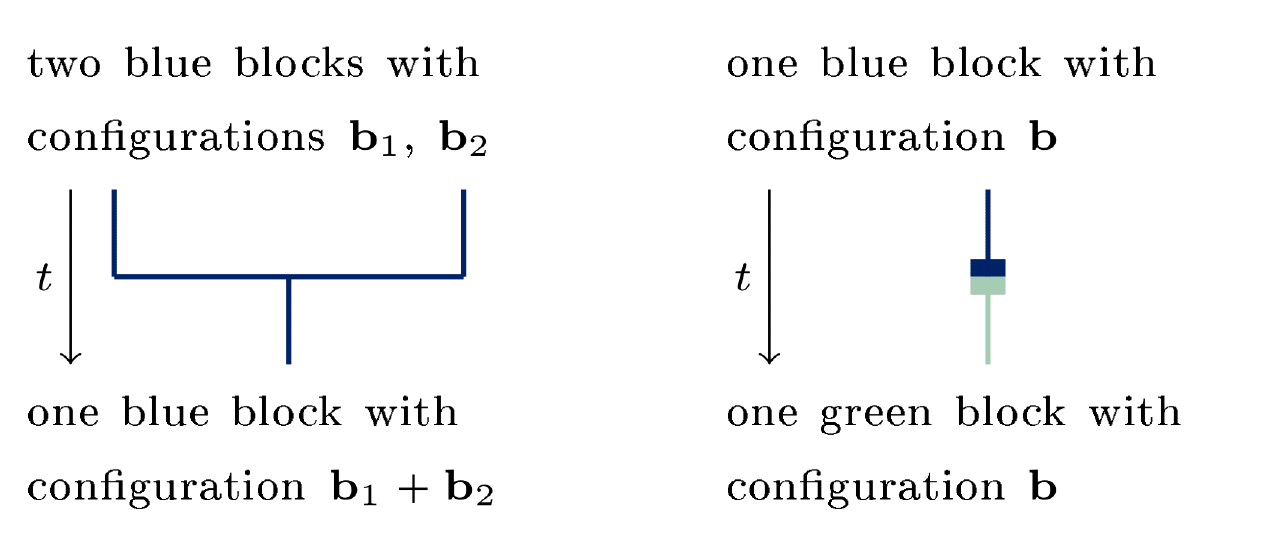}
        }
		\caption{An illustration of the two types of transitions. Left: coalescence (the block color remains the same while the resulting configuration is the sum of the configurations of the two merging blocks). Right: migration (the block color changes while the configuration remains the same).}
		\label{figtransitions}
	\end{figure}
    
For the coalescent part, recall that any pair of block configurations $\bc_1$, $\bc_2$ present in colony $i$ coalesces at rate $\alpha_{i}$, which has the effect on $\bmu^K$ of adding to the coordinate $i$ a mass $1/K$ at $(\bc_1+\bc_2)/s_K$ and removing from the coordinate $i$ a mass $1/K$ at $\bc_1/s_K$ and $\bc_2/s_K$ (see Fig.~\ref{figtransitions} (left)). Since we are scaling time by $1/K$, if we distinguish between the cases where $\bc_1\neq\bc_2$ and $\bc_1=\bc_2$, we obtain
	\begin{align}\label{coagen}
		&A^{K}_{C}  H^{F,\bbf} (\bp)=\frac{K}{2} \sum_{i\in[d]} {\alpha_{i}} \sum_{\bc_{1} \neq \bc_{2} } p_{i}\Big(\Big\{\frac{\bc_1}{s_K}\Big\}\Big)p_{i}\Big(\Big\{\frac{\bc_2}{s_K}\Big\}\Big)\bigg[ F \bigg( \langle \bp, \bbf \rangle + \frac{\Delta_i f(\frac{(\bc_1,\bc_2)}{s_K})}{K}\bigg)- F \left(  \langle \bp, \bbf \rangle  \right)   \bigg]\nonumber \\
		&+\frac12 \sum_{i \in [d]} \alpha_{i}  \sum_{\bc \in \mathbb{N}_{0}^{d}} p_{i}\Big(\Big\{\frac{\bc}{s_K}\Big\}\Big)\Big(Kp_{i}\Big(\Big\{\frac{\bc}{s_K}\Big\}\Big) -1\Big) \bigg[ F \bigg(  \langle \bp, \bbf \rangle + \frac{\Delta_i f(\frac{(\bc,\bc)}{s_K})}{K}\bigg)- F \left(  \langle \bp, \bbf \rangle  \right)   \bigg],
	\end{align}
    where $\Delta_{i}\bbf(\bx_1,\bx_2)= \be_i (f_{i}(\bx_1+\bx_2)-f_{i}(\bx_1)-f_{i}(\bx_2))$.

\subsection{Convergence of the generators}\label{ss:gen-conv}
In this section, we prove a uniform convergence result for the generator $A^K$. Since our ultimate goal is to establish the convergence of the measures $(\bmu^K)_{K\geq 1}$ toward the solution of the coagulation equation~\eqref{SE2} (resp.~\eqref{SE1}), the generator of the limiting object should be given by  
\[
\bar{A}G(\bp)
= \frac{\dd}{\dd t} \, G \circ u(\cdot,\bp)\big|_{t=0},
\]
where $u(\cdot,\bp)$ denotes the solution of Eq.~\eqref{SE2} (resp.~\eqref{SE1}) with initial value $\bp$, $G$ is a smooth enough function, and $\circ$ denotes composition of functions. A straightforward calculation shows that the operator $\bar{A}$ defined above acts on functions $G = H^{F,\bbf} \in \Hs$ via
\(
\bar{A} H^{F,\bbf}(\bp)
= \bar{A}_{M} H^{F,\bbf}(\bp)
+ \bar{A}_{C} H^{F,\bbf}(\bp),\) with
		\begin{align*}
			\bar{A}_{M} H^{F, \bbf} (\bp) &= \sum_{i \neq j} w_{i,j} \left[  \langle f_{j}, p_{i}\rangle  \cdot \partial_{j} F(\langle  \bp, \bbf  \rangle  ) -  \langle f_{i}, p_{i}\rangle  \cdot \partial_{i} F(\langle  \bp, \bbf  \rangle  ) \right]\\
			\bar{A}_{C} H^{F,\bbf} (\bp) &= \frac{1}{2} \sum_{i \in [d]} \left( \int f_{i}(\bc) (p_{i} \star p_{i}) (d\bc)- 2 \langle 1, p_{i} \rangle  \langle f_{i}, p_{i}\rangle         \right) \partial_{i} F(\langle  \bp, \bbf  \rangle  )  \text{.}
		\end{align*}
    The next result formalises the announced convergence result.    

	\begin{proposition}\label{propconvergencegenerators}
		For any $F \in C_2(\Rb^d)$ with bounded second order partial derivatives and $\bbf\coloneqq(f_{i})_{i\in[d]}$ with  $f_{i} \in C_{b}(\mathbb{R}_{+}^d )$, for $i \in [d]$, 
        there is a constant $C(F,\bbf)>0$ such that       
      \begin{align}
       |(A^{K}-\bar{A}) H^{F,\bbf}  (\bp)|  \leq\frac{C(F,\bbf)}{K}\bigg(\sum_{i\in[d]} p_i([N_K]_0/s_K)+p_i([N_K]_0/s_K)^2 \bigg), \quad \bp\in\Ms_f([N_K]_0^d/s_K)^d.
    \label{resultconvergencegenerators}
		\end{align}
	\end{proposition}

	\begin{proof}
		
		Let us start with the migration part of the generator. A second order Taylor's expansion for $F$ at $\langle \bp,\bbf\rangle$ yields  \begin{align*}
			F\bigg(\langle \bp, \bbf \rangle + \frac{\be_j f_j(\frac{\bc}{s_K})-\be_i f_i(\frac{\bc}{s_K})}{K}\bigg) - F(\langle \bp, \bbf \rangle )&=  \partial_{j} F(\langle \bp, \bbf \rangle) \frac{f_j(\frac{\bc}{s_K})}{K}-\partial_{i} F(\langle \bp, \bbf \rangle) \frac{f_i(\frac{\bc}{s_K})}{K} 
            + R_{M}^{K}(\bp,\bc), 
		\end{align*} 
        where $|R_{M}^{K}(\bp,\bc)|\leq C_{F,\bbf}/K^2,$ for a constant $C_{F,\bbf}>0$ depending only on the supremum of $\bbf$ and the second order partial derivatives of $F$.
       Plugging this into Eq.~\eqref{miggen} yields 
        
		\[ A^{K}_{M} H^{F, \bbf} (\bp) =\bar{A}_MH^{F, \bbf}(\bp) + \epsilon_{M}^{K}(\bp),\quad \textrm{with}\quad |\epsilon_M^K(\bp)|\leq \frac{C_{F, \bbf}\wmax}{K}\sum_{i\in[d]} p_i([N_K]_0/s_K), \] 
        and $\wmax\coloneqq\max_{i,j\in[d]}w_{i,j}$.
		
	   Similarly, for the coalescence part of the generator, we use a second order Taylor expansion at $\langle \bp,\bbf\rangle$, to obtain
 \begin{align*}
			F\bigg(\langle \bp, \bbf \rangle + \frac{\Delta_i\bbf(\frac{(\bc_1,\bc_2)}{s_K})}{K}\bigg) - F(\langle \bp, \bbf \rangle )&=  \partial_{i} F(\langle \bp, \bbf \rangle) \frac{f_i(\frac{\bc_{1}+\bc_{2}}{s_K})-f_i(\frac{\bc_{1}}{s_K})-f_i(\frac{\bc_{2}}{s_K}))}{K} 
            + R_{C}^{K}(\bp,\bc_1,\bc_2),
		\end{align*} 
   $|R_{C}^{K}(\bp,\bc_1,\bc_2)|\leq \tilde{C}_{F,\bbf}/K^2,$ for a constant $\tilde{C}_{F,\bbf}>0$ depending only on the supremum of $\bbf$ and the second order derivatives of $F$. Plugging this into Eq.~\eqref{coagen}, we obtain
	\[ A^{K}_{C} H^{F, \bbf} (\bp) =\bar{A}_CH^{F, \bbf}(\bp) + \epsilon_{C}^{K}(\bp),\quad \textrm{with}\quad |\epsilon_C^K(\bp)|\leq \frac{\tilde{C}_{F, \bbf}\amax}{K}\sum_{i\in[d]} \left(p_i([N_K]_0/s_K)\right)^2. \] 
The result follows by combining the bounds for the errors arising from the migration and coalescence parts.
	\end{proof}

\subsection{Moment bounds}\label{ssmomentbounds}
The next result provides bounds for some functionals of the measures $(\mu_i^K)_{i\in[d]}$, which will be used in Section~\ref{secgentight} (see proof of Lemma~\ref{lemmap2}) and Section~\ref{secconvergenceproofs} (see proof of Proposition~\ref{propconvergenceprobability}).

	\begin{lemma}\label{lemlemsecondmoment}
    For any $t\geq 0$, we have
		\begin{equation}  \mathbb{E} \bigg[ \sum_{i \in [d]} \langle \mu_{i}^{K}(t), \|\cdot\|_{1}^{2} \rangle  \bigg]       \leq b^2\Big(\frac{1}{\gamma_{K}} + \amax t\Big) \label{lemsecondmoment}
        \end{equation} and \begin{equation} \mathbb{E} \Bigg[\bigg(\sum_{i \in [d]} \langle \mu_{i}^{K}(t), \|\cdot\|_{1}^{2} \rangle\bigg)^2  \Bigg]       \leq e^{\frac{2\amax}{K}t}b^4\Big(\frac{1}{\gamma_{K}^2} +\frac{2\amax}{\gamma_K}t+ \amax^2t^2 \Big), \label{lemsecondmomentb}
        \end{equation}  where $\amax\coloneqq\max_{i\in[d]} \alpha_i$ and $b$ is defined in \eqref{govers}. 
	\end{lemma}
	\begin{proof} We begin with the proof of \eqref{lemsecondmoment}.
	 Note that for $F_1(\bx)\coloneqq\sum_{i \in [d]} x_{i}$ and $f(\bx)\coloneqq \|\bx\|_{1}^{2}$, we have
		\[ H^{F_1,\bbf}(\bp) =\sum_{i \in [d]} \langle p_i, \|\cdot\|_{1}^{2}\rangle.\]
        Hence $ H^{F,\bbf}\in\Hs$. Moreover, Eq.~\eqref{miggen} yields $A^{K}_{M} H^{F_1,\bbf} (\boldsymbol{\bp}) =0$.
 To deal with coalescences, we group block-configurations in sets of the form $C_\ell^K\coloneqq \{\bc/s_K: \lVert \bc\rVert_1=\ell\}$, and we use Eq.~\eqref{coagen} together with the inequality $\sum_i|a_ib_i|\leq (\sum_i|a_i|)(\sum_i|b_i|)$, to obtain

		\begin{align*}
			A^{K}_{C} H^{F_1,\bbf} (\boldsymbol{\mu}^{K}(t)) &\leq \frac{\amax}{2}  \sum_{\ell_{1}, \ell_{2}\in \mathbb{N}} \bigg(  \sum_{i \in [d]} \mu_{i}^{K}(t, C_{\ell_1}^K)   \bigg)
			\bigg(  \sum_{i \in [d]} \mu_{i}^{K}(t,C_{\ell_2}^K)   \bigg) \left[  \frac{(\ell_{1}+\ell_{2})^{2}- \ell_{1}^{2}-\ell_{2}^{2}}{s_{K}^{2}}  \right] \\			
			&= \frac{\amax}{s_{K}^{2}}   \sum_{\ell_{1}, \ell_{2}\in \mathbb{N}} \bigg(  \sum_{i \in [d]} \mu_{i}^{K}(t, C_{\ell_1}^K)   \bigg)
			\bigg(  \sum_{i \in [d]} \mu_{i}^{K}(t,C_{\ell_2}^K)   \bigg) \ell_1\ell_2 .
		\end{align*} 
		Recalling the definition of $\mu_i^K$ yields
        $$\sum_{i \in [d]}\sum_{\ell\in\Nb}\mu_i^K(t,C_\ell^K)\,\ell=\frac{N_K}{K}.$$
        
        Using this and Fubini yields $A^{K}_{C} H^{F_1,\bbf} (\boldsymbol{\mu}^{K}(t))	\leq b^2\,\amax.$
Thus, setting $\psi_1(t)\coloneqq \mathbb{E}\left[ H^{F_1,\bbf}(\boldsymbol{\mu}^{K}(t))  \right]$, it then follows from Dynkin's formula that \[ \frac{\mathrm{d}}{{\mathrm{d}t}} \psi_1(t) \leq b^2 \amax,    \] 
        Moreover, at time $0$ we have
        \[ \psi_1(0)= \frac{1}{K}\Eb\bigg[ \sum_{i \in [d]} \sum_{j \in [L_i^K(0)]}  \frac{1}{s_{K}^{2}}\bigg] =\frac{\gamma_K}{s_K^2}\leq \frac{b^2}{\gamma_K},  \]         
    and \eqref{lemsecondmoment} follows. \\
    
For \eqref{lemsecondmomentb} we proceed in a similar way, but using the function $F_2(\bx)\coloneq \big(\sum_{i \in [d]} x_{i}\big)^2$. Once more we have $A^{K}_{M} H^{F_2,\bbf} (\boldsymbol{\bp}) =0$. For the coalescence part we have  \begin{align*}
			 A^{K}_{C} H^{F_2,\bbf} (\boldsymbol{\mu}^{K}(t)) &\leq \frac{\amax}{2 K}  \sum_{\ell_{1}, \ell_{2}\in \mathbb{N}}  \sum_{i,j \in [d]} \mu_{i}^{K}(t, C_{\ell_1}^K)
		 \mu_{j}^{K}(t,C_{\ell_2}^K)  \left[ 2 \bigg(\sum_{k \in [d]} \langle \mu_{k}^{K}(t) \|\cdot\|_{1}^{2} \rangle\bigg)\frac{ (\ell_{1}+\ell_{2})^{2}- \ell_{1}^{2}-\ell_{2}^{2}}{s_{K}^{2}} \right.\\
         &\qquad\,\qquad\qquad\qquad\qquad \qquad\qquad\qquad\qquad\qquad\qquad\qquad\qquad\left. + \frac{\big((\ell_{1}+\ell_{2})^{2}- \ell_{1}^{2}-\ell_{2}^{2} \big)^{2}}{s_{K}^{4}}   \right] \\			
            &= 2 b^2 \amax \sum_{i \in [d]} \langle \mu_{i}^{K}(t), \|\cdot\|_{1}^{2} \rangle + \frac{2 \amax}{K} \bigg(\sum_{i \in [d]} \langle \mu_{i}^{K}(t), \|\cdot\|_{1}^{2} \rangle\bigg)^2.
		\end{align*} Thus, setting $\psi_2(t)\coloneqq \mathbb{E}\left[ H^{F_2,\bbf}(\boldsymbol{\mu}^{K}(t))  \right]$, and combining Dynkin's formula with \eqref{lemsecondmoment} yields \[ \frac{\mathrm{d}}{{\mathrm{d}t}} \psi_2(t) \leq 2 \amax b^4 \Big( \frac{1}{\gamma_{K}}+ \amax t \Big) + \frac{2 \amax}{K} \psi_{2}(t).    \] 
        Moreover, at time $0$, we have $\psi_2(0)= b^4/\gamma_{K}^2$. We conclude that  
        \begin{align*}
		    \psi_2(t) &\leq e^{\frac{2 \amax}{K}t} \left[ \psi_{2}(0) + \int_0^t e^{- \frac{2\amax}{K}}\Big(\frac{2 \amax b^4}{\gamma_{K}} +2\amax^2b^4s \Big) \mathrm{d}s    \right],
		\end{align*}
        and \eqref{lemsecondmomentb} follows.
\end{proof}
	    
\subsection{Comparison results}\label{secrescaling}
In this section we will provide a comparison result between the total number of blocks in the structured coalescent $\bPi^K$ and two Kingman coalescents. This result will allow us to get a hand on the order of magnitude of the colony sizes. 	
	
\subsubsection{The coupling}\label{coupling}
Let $L^K\coloneqq (L^K(t))_{t\geq 0}$ be the process that accounts for the total number of blocks in $\bPi^K$, i.e. 
$$L^K(t)\coloneqq \sum_{i \in [d]}L_i^K(t),\quad t\geq 0.$$
We start this section with the announced coupling result.

	\begin{lemma}[Coupling to Kingman]\label{lemdelayedkingman} There is a coupling between the vector-valued process $(L_i^K)_{i \in [d]}$ and other two processes ${\hat L}^{K}$ and $\tilde{L}^{K}$ distributed as the block-counting processes of Kingman coalescents with merger rates per pair of blocks $\amax\coloneqq \max_{i\in[d]}\alpha_i$ and $\amin(d)\coloneqq \min_{i\in[d]}\alpha_i/d^2$, respectively, such that \[ \hat{L}^{K}(t) \leq L^{K}(t) \leq \tilde{L}^{K}(t)\vee (d+1), \quad \text{for all } t \geq 0,  \] 
    and $\hat{L}^{K}(0)=L^{K}(0)= \tilde{L}^{K}(0)=N_K$.
	\end{lemma}
    \begin{proof}
 We first prove deterministic inequalities that will help us to compare the coalescence rates of the three processes. We claim that for any $\bell\coloneqq(\ell_i)_{i \in [d]}\in\Nb_0^d$ with $|\bell|\geq d+1$,
        \begin{equation}\label{ratebounds}
         \amin(d)  |\bell|(|\bell| -1)\leq\sum_{i \in [d]} \alpha_i\ell_{i}(\ell_{i}-1)\leq \amax|\bell|(|\bell| -1).
		\end{equation}    
Indeed, for $\ell\in\Nb_0^d$ with $|\bell|-1\geq d$, Cauchy-Schwarz inequality yields \begin{align*}
			{\sum_{i \in [d]} \ell_{i} (\ell_{i}-1)} = { \|\bell\|_{2}^{2} -|\bell|} \geq { \frac{|\bell|^2}{d} -|\bell|} =  |\bell|\left( \frac{|\bell|}{d} -1\right)=|\bell|(|\bell|-1)\left(\frac{1}{d}-\frac{1-\frac{1}{d}}{|\bell| -1}\right)\geq \frac{1}{d^2}|\bell|(|\bell|-1),
		\end{align*}
and the claimed lower bound for $\sum_{i \in [d]} \alpha_i\ell_{i}(\ell_{i}-1)$ follows. The upper bound is a direct consequence of the inequality $\sum_{i\in[d]} |a_i|^2\leq(\sum_{i\in[d]}|a_i|)^2$. \\
 
 Equipped with \eqref{ratebounds} we construct the announced coupling. We start $(L_i^K)_{i \in [d]}$, ${\hat L}^{K}$ and $\tilde{L}^{K}$ such that $\hat{L}^{K}(0)=L^{K}(0)= \tilde{L}^{K}(0)=N_K$.

Assume that we have constructed them up to the $k$-th transition (transitions refer to the times at which at least one of the three processes makes a jump; we consider time $0$ as the $0$-th transition), which brings $(L_i^K)_{i \in [d]}$, ${\hat L}^{K}$ and $\tilde{L}^{K}$ respectively to states $\ell_*,\bell,\ell^*$ satisfying $\ell_*\leq |\bell|\leq \ell^*$ and $|\bell|\geq d+1$.
The next transition is then constructed as follows. We define 
$$\rho(\bell)\coloneqq K\sum_{i\in[d]}\sum_{j\neq i}\ell_i w_{i,j},\quad \hat{c}(\ell_*)\coloneqq\amax\binom{\ell_*}{2},\quad c(\bell)\coloneqq \sum_{i\in[d]}\alpha_i\binom{\ell_i}{2},\quad \tilde{c}(\ell^*)\coloneqq \amin(d) \binom{\ell^*}{2},$$
and $C(\ell_*,\bell,\ell^*)\coloneqq \max\{\hat{c}(\ell_*),c(\bell),\tilde{c}(\ell^*)\}$.
We refer to $\hat{c}(\ell_*),c(\bell)$ and $\tilde{c}(\ell^*)$ as coalescence rates of $\ell_*$, $\bell$, $\ell^*$, respectively. Then after an exponential time with parameter $\lambda(\ell_*,\bell,\ell^*)\coloneqq \rho(\bell)+\max\{\hat{c}(\ell_*),c(\bell),\tilde{c}(\ell^*)\}$
\begin{enumerate}
    \item with probability $\rho(\bell)/\lambda(\ell_*,\bell,\ell^*)$ a migration event takes place (affecting only $\bell$). More precisely,
    $$\bell\to \bell+\be_j-\be_i,\quad\textrm{with probability $K\ell_iw_{i,j}/\rho(\bell)$},$$
    $\ell_*,\ell^*$ (and $|\bell|$) remain unchanged.
    \item with probability $C(\ell_*,\bell,\ell^*)/\lambda(\ell_*,\bell,\ell^*)$ a coalescence event takes place. The coalescence will always produce a transition among states $\ell_*,\bell,\ell^*$ whose coalescence rate is equal to $C(\ell_*,\bell,\ell^*)$; the state associated with the second highest coalescence rate will only be modified with some probability and only then will the state with the lowest coalescence rate be affected with some probability.
    
       There are $6$ possible cases to consider that correspond to the different orderings of the coalescence rates of $\ell_*,\bell,\ell^*$. We explain in detail one case; the others are analogous. We consider the case where $\hat{c}(\ell_*)\geq c(\bell)\geq \tilde{c}(\ell^*)$. In this case, the transition 
    $$\ell_*\to\ell_*-1$$
    takes place. In addition, with probability $c(\bell)/\hat{c}(\ell_*)$ the state $\bell$ transitions. The transition affects only one of its coordinates; it is the $i$-th coordinate with probability $\alpha_i\ell_i(\ell_i-1)/2c(\bell)$, in which case the transition $$\bell\to\bell-\be_i$$
    takes place. Finally, only if $\bell$ changes, the transition $\tilde{c}(\ell^*)/c(\bell)$ 
    $$\ell^*\to\ell^*-1$$
    occurs with probability $\tilde{c}(\ell^*)/c(\bell)$ ; otherwise, $\ell^*$ remains unchanged.
    \end{enumerate}
As soon as $|\bell|$ goes below $d+1$, we continue to carry out the coupling between $\hat{L}^K$ and $L^K$; the process $\tilde{L}^K$ can then be further constructed independently.
    
    Note that, thanks to \eqref{ratebounds}, if $\ell_*=|\bell|$, then $\hat{c}(\ell_*)\geq c(\bell)$. Similarly, if $\ell^*=|\bell|\geq d+1$, then $\tilde{c}(\ell_*)\leq c(\bell)$. This, together with the fact that the transitions only decrease the size of the affected states by one, implies that the transitions will preserve the ordering $\ell_*\leq |\bell|\leq \ell^*$. 
    The proof is achieved by noticing that the so-constructed processes have the desired distributions.
  \end{proof}
\subsubsection{Kingman bounds}  
The following result on Kingman's coalescent will be helpful in many proofs. We provide a short proof for completeness. 
\begin{lemma}[Moment bounds]\label{moki}
Let $(L_\rho^K(t))_{t\geq 0}$ be the block-counting process of a Kingman coalescent with merger rate $\rho$, which started with $N_K$ blocks at time $0$. Then, for any $p \geq 1$, we have
$$\Eb[(L_\rho^K(t))^p]\leq{\bigg(\frac{1}{N_K^{1/p}}+\frac{\rho}{4p} t\bigg)}^{-p},\qquad t\geq 0.$$
\end{lemma}
\begin{proof}
Let $\As$ denote the generator of $L_\rho^K$ and $\psi(n)=n^p$. Clearly,
$$\As\psi(n)=\rho\frac{n(n-1)}{2} ((n-1)^p-n^p),\qquad n\geq 1.$$
Using that, for $n\geq 1$, $n^p-(n-1)^p\geq n^{p-1}$ and $n(n-1)\geq n^2/2$, we get $$\As\psi(n)\leq -\frac{\rho}{4} \psi(n)^{1+1/p}.$$
Combining this with Dynkin's formula and Jensen inequality, we get
$$\Eb[\psi(L_\rho^K(t))]\leq \psi(N_K)-\frac{\rho}{4}\int_0^t\Eb[\psi(L_\rho^K(s))^{1+1/p}]\dd s\leq \psi(N_K)-\frac{\rho}{4}\int_0^t\Eb[\psi(L_\rho^K(s))]^{1+1/p}\dd s.$$
Setting $v(t)=\Eb[\psi(L_\rho^K(t))]$, the previous inequality implies that
$$v'(s)\leq -\frac{\rho}{4}v(s)^{1+1/p}.$$
Dividing both sides by $v(s)^{1+1/p}$ and integrating both sides of the resulting inequality between $0$ and $t$ and rearranging terms yields the result.
\end{proof}

\section{Tightness and characterization of accumulation points}\label{secgentight}
Having proved convergence of the generators in Section~\ref{ss:gen-conv}, we now turn to the proof of tightness of the sequence $\{\bmu^{K}\}_{K\geq 1}$ (see Proposition~\ref{proptightness} in Section~\ref{ss:tighness}) and to the characterization of its accumulation points (see Proposition~\ref{propaccum} in Section~\ref{ss:charac}), thereby paving the way to the desired convergence result.

\subsection{Tightness}\label{ss:tighness}
	The next proposition establishes the announced tightness of the sequence $\{\bmu^{K}\}_{K\geq 1}$.
    \begin{proposition} \label{proptightness} Let $t_0=0$ in the critical sampling case and $t_0=\varepsilon>0$ (with $\varepsilon>0$ fixed, but arbitrary) in the large sampling case. For any $T>t_0$, the sequence of measure-valued processes $\left\{ (\bmu^K(t))_{t\in[t_0,T]}   \right\}_{K\geq 1}$ is tight in $D([t_0,T], (\mathcal{M}_{f} (\mathbb{R}_{+}^{d})^d,w))$, where $(w)$ stands for the weak topology,
	\end{proposition}
The proof of this result follows a classical strategy and relies on several key ingredients. For clarity, we first establish the intermediate results associated with these ingredients and explain how they combine to yield the desired result. The remainder of the section is then devoted to proving these intermediate results.
\subsubsection{The (skeleton of the) proof of Proposition \ref{proptightness}}	
To prove tightness of the sequence \(\{\bmu^K\}_{K\geq 1}\), we will show that the sequence of product measures \[ \omu\coloneqq \bigotimes_{i \in [d]} \mu_{i}^{K},\qquad K\geq 1,\]
is tight in the weak topology. The tightness of $\{\bmu^K\}_{K\geq 1}$ follows then as an application of the continuous mapping theorem (see \cite[Thm.~2.7]{billingsley2013convergence}).

With this in mind, we introduce the following notation. Let $\mathcal{F}$ denote the space of functions $f_{\otimes} \in C_{b}(\mathbb{R}_+^{d^{2}})$ of the form \[ f_{\otimes}(\bx_{1}, \dots, \bx_{d}) = \sum_{k \in [n]} \prod_{i \in [d]} f_{k,i}(\bx_{i}), \qquad \bx_{1}, \dots, \bx_{d} \in \mathbb{R}_{+}^{d}, \] 
for  some $n\in\Nb$ and $f_{k,i} \in C_{b}(\mathbb{R}_{+}^{d})$ for $k\in[n]$, $i\in[d]$. According to the Stone-Weierstrass theorem, $\Fs$ is dense in $C_{0}(\mathbb{R}^{d^{2}})$ the set of continuous functions vanishing at infinity.

To prove the tightness of $\{\omu\}_{K\geq 1}$ in $D([t_0,T], (\mathcal{M}_{f}(\mathbb{R}_{+}^{d^2}), w))$, we follow a standard approach (see, e.g. \cite{FouMel, tran2014ballade, Tran}). In our setting, this approach amounts to establishing the following three lemmas.
\begin{lemma}[Tightness of integrals]\label{lemmap1}
    Let $t_0=0$ in the critical sampling case and $t_0=\varepsilon>0$ in the large sampling case. For every function $f_{\otimes} \in \mathcal{F}$ and $T>t_0$, the sequence $\{ (\langle \omu(t),f_{\otimes} \rangle)_{t\in[t_0,T]}\}_{K\geq 1}$ is tight in $D([t_0,T], \mathbb{R})$.
	\end{lemma}
\begin{lemma}[Uniform-in-time moment bound]\label{lemmap2}
    Let $t_0=0$ in the critical sampling case and $t_0=\varepsilon>0$ in the large sampling case.	We have \[\limsup_{K \rightarrow \infty} \mathbb{E} \bigg[  \sup_{t \in [t_0,T]} \langle  \mu_{\otimes}^{K}(t), \| \cdot \|_{2}^{2}\rangle   \bigg] < \infty . \] 
	\end{lemma}
\begin{lemma}[Continuity of accumulation points]\label{lemmap3}
     Let $t_0=0$ in the critical sampling case and $t_0=\varepsilon>0$ in the large sampling case. Any accumulation point $\mu_{\otimes}^{\infty} $ of $\left\{ \mu_{\otimes}^{K} \right\}_{K \geq 1}$ in $D([t_0,T], (\mathcal{M}_{f} (\mathbb{R}_{+}^{d^{2}}), w))$ belongs to $C([t_0,T], (\mathcal{M}_{f} (\mathbb{R}_{+}^{d^{2}}), w))$.
	\end{lemma}
Assuming these three lemmas, Proposition~\ref{proptightness} follows directly from \cite[Thm.~1.1.8]{tran2014ballade}.

This strategy is analogous to the one used in \cite[Thm.~7.4]{lambert2020coagulation}, although the authors there employ a version of Lemma \ref{lemmap3} formulated for accumulation points in $D([t_{0},T], (\mathcal{M}_{f}(\mathbb{R}_{+}^{d^{2}}), v))$, where $v$ denotes the vague topology. According to Roelly’s criterion (see \cite{Roelly}), Lemma~\ref{lemmap1} alone already ensures tightness in $D([t_{0},T], (\mathcal{M}_{f}(\mathbb{R}_{+}^{d^{2}}), v))$. Lemmas~\ref{lemmap2} and~\ref{lemmap3} then allow one to upgrade vague convergence to weak convergence (see \cite[Lemma~1.1.9]{tran2014ballade}).

	\begin{remark}\label{ballade}
		The bound in Lemma \ref{lemmap2} slightly differs from the property stated in \cite[Thm.~1.1.8]{tran2014ballade}, which reads as  
        \begin{equation}\label{P2b}
            \lim_{k \rightarrow \infty} \limsup_{K \rightarrow \infty} \sup_{t \in [0,T]} \langle \mu_{\otimes}^{K}(t), \varphi_{k}(\|\cdot\|_{2}) \rangle = 0 \text{, in probability,}         
            \end{equation}
        where $\varphi_{k}$ is a function satisfying $\mathds{1}_{\{ \|\bx\|_{2} \geq k\}} \leq \varphi_{k}(\|\bx\|_{2}) \leq  \mathds{1}_{\{ \|\bx\|_{2} \geq k-1 \}}  .$		
     Note that Markov's inequality implies that, for all $t\geq t_0$,
        \begin{align*}
			\langle \mu_{\otimes}^{K}(t), \varphi_{k}(\|\cdot\|_{2}) \rangle \leq \langle \mu_{\otimes}^{K}(t), \mathds{1}_{\{ \|\cdot\|_{2} \geq k-1 \}} \rangle = \langle \mu_{\otimes}^{K}(t), \mathds{1}_{\{ \|\cdot\|_{2}^{2} \geq (k-1)^2 \}} \rangle \leq \frac{1}{(k-1)^2} \langle \mu_{\otimes}^{K}(t), \|\cdot\|_{2}^{2} \rangle. 
		\end{align*} 
        In particular, we have \begin{align*}
			\mathbb{E}\bigg[ \sup_{t \in [t_0,T]} \langle \mu_{\otimes}^{K}(t), \varphi_{k}(\|\cdot\|_{2}) \rangle  \bigg] \leq  \frac{1}{(k-1)^2} \mathbb{E}\bigg[ \sup_{t \in [t_0,T]} \langle \mu_{\otimes}^{K}(t), \|\cdot\|_{2}^{2} \rangle \bigg]. 
		\end{align*} Another application of Markov's inequality shows that, under the bound in Lemma \ref{lemmap2}, condition \eqref{P2b} holds.
	\end{remark}
		\begin{remark}
	In view of our proof strategy, it will be useful to notice that (recall \eqref{eqH}), we have for $f_{\otimes} \in \mathcal{F}$ as above and $ \pi_{[d]}(\bx)\coloneqq \prod_{i \in [d]} x_{i}$  \begin{equation} \langle \mu_{\otimes}^{K}(t), f_{\otimes} \rangle = \sum_{k \in [n]} H^{\pi_{[d]},\bbf_{k}}(\boldsymbol{\mu}^{K}(t)). \label{eqtransproducgenerator}  
\end{equation}
	\end{remark}
	
The remainder of Section~\ref{ss:tighness} is devoted to proving Lemmas~\ref{lemmap1}, \ref{lemmap2} and \ref{lemmap3}.
    
\subsubsection{On the proof of Lemma \ref{lemmap1}}
The proof of Lemma~\ref{lemmap1} is based on a classical result of Aldous and Rebolledo (see \cite{aldous1978stopping, joffe1986weak}), which relies on two main ingredients:
(i) tightness of integrals of functions with respect to $\omu$ at fixed times (see Lemma \ref{AR1}); and 
(ii) bounds on the martingale and finite-variation parts of $H(\bmu^{K})$ for an appropriate class of test functions $H$ (see Lemma \ref{AR2}). \\

We begin with an intermediate result that provides bounds on the increments of the quadratic variation of $H(\bmu^{K})$ and on $A^{K}H$ along $\bmu^{K}$ for a suitable class of test functions $H$ (see Lemma~\ref{lemmaquadraticvariation}). This result will be crucial for establishing the first ingredient and will also be used in Section~\ref{ss:charac}. To this end, we introduce the functions ${\pi}_{I} : \Rb^{d} \to \Rb$, for $I \subseteq [d]$, defined by
\({\pi}_{I}(\bx) = \prod_{j \in I} x_{j}.\)
 \begin{lemma}[Quadratic variation] \label{lemmaquadraticvariation} Let $t_0=0$ in the critical sampling case and $t_0=\varepsilon>0$ in the large sampling case. For any  $\bbf\coloneqq (f_{i})_{i \in [d]}$, $f_{i} \in C_{b}(\mathbb{R}_{+}^{d})$ and 
 $I \subseteq [d]$, there is a constant $C({I},\bbf)>0$ 
 and decreasing processes $(\Is_K(s))_{s\geq t_0}$ and $(\Is_K^*(s))_{s\geq t_0}$ satisfying that
 $$\sup_{K \geq 1,s\geq t_0}\Eb[\Is_K(s)]<\infty,\quad \sup_{K \geq 1,s\geq t_0}\Eb[\Is_K^*(s)]<\infty,$$
and such that for any $t \geq u \geq t_0$
  \begin{align}
		\left| \langle H^{{\pi}_{I},\bbf}(\boldsymbol{\mu}^{K}) \rangle _{t}-\langle H^{{\pi}_{I},\bbf}(\boldsymbol{\mu}^{K}) \rangle _{u}    \right| &\leq  \frac{C({I},\bbf)}{K} \Is_K(u)(t-u), \label{eqquadraticvar}\\
        |A^{K} H^{{\pi}_{I}, \bbf} (\boldsymbol{\mu}^{K}(s))|& \leq {C}({I},\bbf)\Is_K^*(s), \label{eqgenbound}	\end{align}	 where $ \langle \cdot \rangle _{t}$ stands for the quadratic variation at time $t$.
   	\end{lemma}

	\begin{proof} We begin with the proof of \eqref{eqquadraticvar}. Recall the definition of $s_{K}$ from Eq.~\eqref{eqdefsk}. We may then write
		\begin{align*}
			\langle H^{\pi_{I},\bbf}(\boldsymbol{\mu}^{K})\rangle _{t}-\langle H^{\pi_{I},\bbf}(\boldsymbol{\mu}^{K})\rangle _{u} = \int_{u}^{t} I_{s}^{M} \mathrm{d}s + \int_{u}^{t} I_{s}^{C}\, \mathrm{d}s
		\end{align*} where
\begin{align*}
&I^{M}(s,K) \coloneqq \sum_{i \neq j}  w_{i,j} \sum_{\bc \in \mathbb{N}_{0}^{d} } K \mu^{K}_{i}\Big(s, \Big\{\frac{\bc}{s_{K}}\Big\}\Big) \bigg[\pi_{I} \bigg( \langle  \boldsymbol{\mu}^{K}(s), \bbf\rangle  + \frac{\Delta_{i,j}\bbf(\frac{\bc}{s_{K}})}{K}    \bigg)  - \pi_{I} \left( \langle \boldsymbol{\mu}^{K}(s), \bbf \rangle \right)  \bigg]^{2},\\
&I^C(s, K) \coloneqq  \sum_{i \in [d]} \frac{\alpha_i K}{2}\sum_{\bc_{1} \neq \bc_{2} } \mu^{K}_{i}\Big(s, \Big\{ \frac{\bc_{1}}{s_{K}} \Big\}\Big) \mu^{K}_{i}\Big(s, \Big\{\frac{\bc_{2}}{s_{K}}\Big\}\Big) \\&\qquad \qquad \qquad \qquad \qquad \qquad \qquad \qquad \qquad \bigg[\pi_{I} \bigg( \langle \boldsymbol{\mu}^{K}(s), \bbf \rangle + \frac{\Delta_i\bbf (\frac{\bc_1}{s_{K}},\frac{\bc_2}{s_{K}} )}{K}  \bigg)  -  \pi_{I} \left( \langle  \boldsymbol{\mu}^{K}(s), \bbf \rangle  \right)   \bigg]^2  \\
			&+ \sum_{i \in [d]}\frac{\alpha_i}{2}  \sum_{\bc \in \mathbb{N}_{0}^{d} } \mu^{K}_{i}\Big(s,\Big\{ \frac{\bc}{s_{K}}\Big\}\Big) \Big(K\mu^{K}_{i}\Big(s, \Big\{ \frac{\bc}{s_{K}} \Big\}\Big) -1\Big) \\
            &\qquad \qquad \qquad \qquad \qquad \qquad \qquad \qquad \qquad \bigg[  \pi_{I} \bigg( \langle \boldsymbol{\mu}^{K}(s), \bbf \rangle + \frac{\Delta_i\bbf ( \frac{\bc}{s_{K}}, \frac{\bc}{s_{K}})}{K} \bigg) -\pi_{I} \left(  \langle \boldsymbol{\mu}^{K}(s), \bbf \rangle \right)   \bigg]^2.	
\end{align*}
  
     Using that \begin{align*}
         \pi_{I} &\bigg( \langle  \boldsymbol{\mu}^{K}(s), \bbf\rangle  + \frac{\Delta_{i,j}\bbf(\frac{\bc}{s_{K}})}{K}    \bigg)  - \pi_{I} \left( \langle \boldsymbol{\mu}^{K}(s), \bbf \rangle \right)\\
         &\leq \max_{h\in I} \frac{\|f_{h}\|_{\infty}^{|I|} }{K} \bigg( \prod_{h \in I \setminus \{i\}} \langle \mu_{h}^{K}(s), 1 \rangle +  \prod_{h \in I \setminus \{j\}} \langle \mu_{h}^{K}(s), 1 \rangle + \frac{1}{K}  \prod_{h \in I \setminus \{i, j\}} \langle \mu_{h}^{K}(s), 1 \rangle   \bigg)
     \end{align*} and \begin{align*}
         \pi_{I} \bigg( \langle \boldsymbol{\mu}^{K}(s), \bbf \rangle + \frac{\Delta_i\bbf ( \frac{\bc_1}{s_{K}}, \frac{\bc_2}{s_{K}} )}{K}  \bigg)  -  \pi_{I} \left( \langle  \boldsymbol{\mu}^{K}(s), \bbf \rangle  \right) \leq \max_{h\in I} \frac{3 \|f_{h}\|_{\infty}^{|I|} }{K} \prod_{h \in I \setminus \{i\}} \langle \mu_{h}^{K}(s), 1 \rangle,
     \end{align*} since $\langle \mu_{i}^{K}(s), 1 \rangle \leq \frac{L^{K}(\frac{s}{K})}{K}$, we see that    \begin{align*}
		|I^{M}(s,K)| &\leq \frac{C_M({I},\bbf)}{K} \left( \left( \frac{L^{K}(\frac{s}{K})}{K} \right)^{2|I|-1} + \frac{1}{K} \left( \frac{L^{K}(\frac{s}{K})}{K} \right)^{(2|I|-3)\mathds{1}_{\{|I|>1\}}}  \right) \\
        \text{and}\quad |I^{C}(s,K)| &\leq \frac{C_{C}({I},\bbf)}{K} \left( \frac{L^{K}(\frac{s}{K})}{K} \right)^{2|I|} ,	
         \end{align*}
	   for some constants $C_M(I,\bbf),C_C({I},\bbf)>0$ depending only on the functions $\pi_{I},\bbf$. Since the moments of $L^{K}(s/K)/K$ are bounded due to Lemma~\ref{lemdelayedkingman} and Lemma~\ref{moki}, the Eq.~\eqref{eqquadraticvar} follows. We obtain Eq.~\eqref{eqgenbound} by repeating the same arguments without the squares.    
	\end{proof}
Let us now turn to the first ingredient, which is the content of the next lemma.
\begin{lemma}\label{AR1}
In the critical (resp. large) sampling case,  for every fixed $t\geq 0$ (resp. $t>0$), the sequence $\{ \langle \mu_{\otimes}^{K}(t),f_{\otimes} \rangle  \}_{K \geq 1}$ is tight.
\end{lemma}
	\begin{proof}
	We first note that
    \begin{align*}
			|\langle \mu_{\otimes}^{K}(t),f_{\otimes} \rangle  | = \left| \sum_{k \in [n]} \prod_{i \in [d]} \langle \mu_{i}^{K}(t), f_{k,i }\rangle \right| \leq n\max_{k \in [n], j \in [d]}\lVert f_{k,j}\rVert \prod_{i \in [d]}  \langle \mu_i^K(t),1\rangle.
		\end{align*} 
        In the critical sampling case, we have $\langle \mu_i^K(t),1\rangle\leq \gamma_K$ and $\gamma_K$ is bounded, and hence, the result follows. In the large sampling case, we use that
        $$\prod_{i \in [d]}\langle \mu_i^K(t),1\rangle\leq \frac{(L^K(\frac{t}{K}))^d}{K^d}.$$
        Using this, Markov inequality, Lemma~\ref{lemdelayedkingman} and Lemma~\ref{moki}, we get for $t>0$
        \begin{align*}
         \Pb(|\langle \mu_{\otimes}^{K}(t),f_{\otimes} \rangle  |>M)&\leq c_1\frac{\Eb[(L^K(\frac{t}{K}))^d]}{MK^d} \leq  c_1 \frac{\Eb[({\tilde L}^K(\frac{t}{K}) {\vee (d+1)} )^d]}{MK^d} \leq \frac{c_2}{t^d M}, 
        \end{align*}
        for some constants $c_1,c_2>0$. The result follows.
	\end{proof}
  The second ingredient is provided by the following result.
    \begin{lemma}\label{AR2}
Let $t_0=0$ in the critical sampling case and $t_0=\varepsilon>0$ in the large sampling case.  Let $\bbf\coloneqq (f_{i})_{i \in [d]}$, $f_{i} \in C_{b}(\mathbb{R}_{+}^{d})$, and  $I \subseteq [d]$. Define
      		\begin{align*}
            B_{t}^{K,I,\bbf}& \coloneqq   H^{\pi_{I},\bbf} (\boldsymbol{\mu}^{K}(t_0)) + \int_{t_0}^{t} A^{K} H^{\pi_{I},\bbf} (\boldsymbol{\mu}^{K}(s)) \mathrm{d}s\\
			\text{and} \quad M_{t}^{K,I,\bbf}& \coloneqq  H^{\pi_{I},\bbf}(\boldsymbol{\mu}^{K}(t)) - H^{\pi_{I},\bbf} (\boldsymbol{\mu}^{K}(t_0))  - \int_{t_0}^{t} A^{K} H^{\pi_{I},\bbf} (\boldsymbol{\mu}^{K}(s)) \mathrm{d}s.
		\end{align*}
   Then, for any $\delta >0$ and any pair of stopping times $(\tau, \sigma)$ such that $ t_0 \leq \tau \leq \sigma \leq \tau + \delta \leq T$, there are constants $\tilde{c}_1({I}, \bbf, t_0), \tilde{c}_2({I}, \bbf, t_0) >0$ such that \begin{align*} \mathbb{E}[| M_{\sigma}^{K,I,\bbf} - M_{\tau}^{K,I,\bbf}|] \leq \tilde{c}_1({I}, \bbf, t_0) \sqrt{\frac{\delta}{K}} \quad\text{and}\quad \mathbb{E} [| B_{\sigma}^{K,I,\bbf} - B_{\tau}^{K,I,\bbf}|] &\leq \tilde{c}_2(\pi_{I}, \bbf, t_0) \delta.  
   \end{align*}
  In particular, the two quantities are bounded from above by a function of $\delta$ that goes to $0$ as $\delta \rightarrow 0$.    
   \end{lemma}

\begin{proof}
        Using Jensen's inequality and the Martingale property, we get
		\begin{align*}
			\mathbb{E} \left[\left|M_{\sigma}^{K,I,\bbf}- M_{\tau}^{K,I,\bbf} \right| \right]^2 &\leq \mathbb{E} \left[ \left(M_{\sigma}^{K,I,\bbf}- M_{\tau}^{K,I,\bbf} \right)^{2} \right] = \mathbb{E} \left[ \left(M_{\sigma}^{K,I,\bbf}\right)^{2}- \left(M_{\tau}^{K,I,\bbf}\right)^{2}  \right].
		\end{align*} In addition, according to Eq.~\eqref{eqquadraticvar} from Lemma~\ref{lemmaquadraticvariation}, we find
		\begin{align*}
			\mathbb{E} \left[ \left(M_{\sigma}^{K,I,\bbf}\right)^{2}- \left(M_{\tau}^{K,I,\bbf}\right)^{2}  \right]
			=& \mathbb{E} \left[  \langle H^{\pi_{I},\bbf}(\boldsymbol{\mu}^{K}) \rangle _{\sigma}-\langle H^{\pi_{I},\bbf}(\boldsymbol{\mu}^{K}) \rangle _{\tau}     \right]
           \leq \frac{C({I},\bbf)}{K} \sup_{K \geq 1,s\geq t_0} \mathbb{E} [ \Is_K(s)] \delta.           
            \end{align*}         
                Summarizing, taking the square-root, 
		\begin{align*}
            \mathbb{E}[|M_{\sigma}^{K,I,\bbf}-M_{\tau}^{K,I,\bbf}|]
           &\leq \sqrt{ \frac{C({I},\bbf)}{K} \sup_{K \geq 1,s\geq t_0}\mathbb{E} [ \Is_K(s)] \delta} \leq \sqrt{C({I},\bbf) \sup_{K \geq 1,s\geq t_0} \mathbb{E} [ \Is_K(s)] \delta} \underset{\delta \rightarrow 0}{\longrightarrow} 0, 
		\end{align*}
		uniformly in  $K$, $\sigma$ and  $\tau$. For the second expectation, we use Eq.~\eqref{eqgenbound} to find    
        \begin{align*}
			     \mathbb{E}[|B_{\sigma}^{K,I,\bbf}-B_{\tau}^{K,I,\bbf}|] &\leq \mathbb{E} \left[ \int_{\tau}^{\sigma} |A^{K} H^{\pi_{I}, \bbf} (\boldsymbol{\mu}^{K}(s))| \mathrm{d}s     \right] \leq {C}({I},\bbf)  \sup_{K \geq 1,s\geq t_0}\mathbb{E}[\Is_K^*(s)] \delta \underset{\delta \rightarrow 0}{\longrightarrow} 0,
         \end{align*} uniformly in $K$, $\sigma$ and $\tau$. \qedhere
\end{proof}
We conclude this section with the proof of Lemma~\ref{lemmap1}.
 \begin{proof}[Proof of Lemma \ref{lemmap1}]
	The result follows as a consequence of Lemma~\ref{AR1} and Lemma~\ref{AR2} and the Aldous--Rebolledo tightness criterium (\cite{aldous1978stopping}, \cite{joffe1986weak}).
\end{proof}

\subsubsection{On the proof of Lemma \ref{lemmap2}}
We now establish the announced uniform-in-time moment bounds.
	\begin{proof}[Proof of Lemma \ref{lemmap2}]
    	For the critical sampling, we would like to make use of Eq.~\eqref{lemsecondmoment} from Lemma~\ref{lemlemsecondmoment}. To that end, once again since $\langle\mu_{i}^{K}(t), 1 \rangle  \leq b$, we find (with abuse of notation)
		\begin{align*} \mathbb{E} \bigg[  \sup_{t \in [0,T]} \langle \mu_{\otimes}^{K}(t),  \|\cdot\|_{2}^{2} \rangle   \bigg]
			&=  \mathbb{E} \bigg[  \sup_{t \in [0,T]} \sum_{i \in [d]} \langle \mu_{i}^{K}(t),  \|\cdot\|_{2}^{2} \rangle \prod_{n \in [d] \setminus \{i\}}^{} \langle  \mu_{n}^{K}(t), 1 \rangle  \bigg] \\
            &\leq b^{d-1}  \mathbb{E} \bigg[  \sup_{t \in [0,T]} \sum_{i \in [d]}  \langle \mu_{i}^{K}(t), \|\cdot\|_{2}^{2} \rangle    \bigg].		 
		\end{align*} Note that on the left hand side we integrate $\mu_{\otimes}^{K}(t)$ against the 2-Norm on $\mathbb{R}^{d^{2}}$, whereas on the right hand side we integrate  $\mu_{i}^{K}(t)$ against the 2-Norm on $\mathbb{R}^{d}$. Further, since the sum of the squares is smaller than the square of the sum:
		\[ \mathbb{E} \bigg[  \sup_{t \in [0,T]} \sum_{i \in [d]}  \langle \mu_{i}^{K}(t), \|\cdot\|_{2}^{2} \rangle    \bigg] \leq \mathbb{E} \bigg[  \sup_{t \in [0,T]} \sum_{i \in [d]}  \langle \mu_{i}^{K}(t), \|\cdot\|_{1}^{2} \rangle    \bigg].   \] Last, since the block sizes are only increasing over time, and again the sum of the squares is smaller than the square of the sum, Eq.~\eqref{lemsecondmoment} already yields  \[  \mathbb{E} \bigg[  \sup_{t \in [0,T]} \langle \mu_{\otimes}^{K}(t),  \|\cdot\|_{2}^{2} \rangle   \bigg] \leq b^{d-1} \mathbb{E} \bigg[ \sum_{i \in [d]}  \langle \mu_{i}^{K}(T), \|\cdot\|_{1}^{2} \rangle    \bigg] \leq b^{d+1} \Big( \frac{1}{\gamma_{K}} +T \Big),   \] because we assumed $\alpha_i =1$ for all $i \in [d]$ throughout the section. Since $ \gamma_{K}= N_{K}/K $, we conclude:  \[\limsup_{K \rightarrow \infty} \mathbb{E} \bigg[  \sup_{t \in [0,T]} \langle  \mu_{\otimes}^{K}(t), \|\bx\|_{2}^{2}\rangle   \bigg] \leq \limsup_{K \rightarrow \infty} b^{d+1} \Big( \frac{1}{\gamma_{K}} +T\Big)  < \infty . \] 
        
        The proof for the large sampling follows the same general idea, only that here, we make use of Eq.~\eqref{lemsecondmomentb}. More explicitly, this time we find (again with abuse of notation) \begin{align*} \mathbb{E} \bigg[  \sup_{t \in [\varepsilon,T]} \langle \mu_{\otimes}^{K}(t),  \|\cdot\|_{2}^{2} \rangle   \bigg]
			&=  \mathbb{E} \bigg[  \sup_{t \in [\varepsilon,T]} \sum_{i \in [d]} \langle \mu_{i}^{K}(t),  \|\cdot\|_{2}^{2} \rangle \prod_{n \in [d] \setminus \{i\}}^{} \langle  \mu_{n}^{K}(t), 1 \rangle  \bigg] \\
            &\leq  \mathbb{E} \bigg[  \sup_{t \in [\varepsilon,T]} \sum_{i \in [d]} \langle \mu_{i}^{K}(t),  \|\cdot\|_{2}^{2} \rangle^2   \bigg] + \mathbb{E} \bigg[  \sup_{t \in [\varepsilon,T]} \sum_{i \in [d]} \bigg(  \prod_{n \in [d] \setminus \{i\}}^{} \langle  \mu_{n}^{K}(t), 1 \rangle \bigg)^2   \bigg].		 
		\end{align*}  Note that (again) on the left hand side we integrate $\mu_{\otimes}^{K}(t)$ against the 2-Norm on $\mathbb{R}^{d^{2}}$, whereas on the right hand side we integrate  $\mu_{i}^{K}(t)$ against the 2-Norm on $\mathbb{R}^{d}$. For the first expectation on the right hand side, we again use that the sum of squares is smaller than the square of the sum and that blocks are only increasing over time: \begin{align*}
		    \mathbb{E} \bigg[  \sup_{t \in [\varepsilon,T]} \sum_{i \in [d]} \langle \mu_{i}^{K}(t),  \|\cdot\|_{2}^{2} \rangle^2   \bigg] \leq \mathbb{E} \bigg[  \bigg(\sum_{i \in [d]} \langle \mu_{i}^{K}(T),  \|\cdot\|_{1}^{2} \rangle \bigg)^2   \bigg].
		\end{align*} We may therefore apply Eq.~\eqref{lemsecondmomentb} to find \begin{align}
		    \mathbb{E} \bigg[  \sup_{t \in [\varepsilon,T]} \sum_{i \in [d]} \langle \mu_{i}^{K}(t),  \|\cdot\|_{2}^{2} \rangle^2   \bigg] \leq e^{\frac{2\amax}{K}T}b^4\Big(\frac{1}{\gamma_{K}^2} +\frac{2\amax}{\gamma_K}T+ \amax^2T^2 \Big). \label{eqsecondmomentfirstexp}
		\end{align} For the second expectation, since the number of blocks is decreasing over time, we first find         \begin{align*}
            \mathbb{E} \bigg[  \sup_{t \in [\varepsilon,T]} \sum_{i \in [d]} \bigg(  \prod_{n \in [d] \setminus \{i\}}^{} \langle  \mu_{n}^{K}(t), 1 \rangle \bigg)^2   \bigg] \leq d  \mathbb{E} \bigg[ \bigg(  \frac{L^{K}(\frac{\varepsilon}{K})}{K} \bigg)^{2(d-1)}   \bigg].
        \end{align*} Moreover, according to Lemmas~\ref{lemdelayedkingman} and~\ref{moki} \begin{align}   \mathbb{E} \bigg[ \bigg(  \frac{L^{K}(\frac{\varepsilon}{K})}{K} \bigg)^{2(d-1)}   \bigg]  \leq \frac{1}{K^{2(d-1)}}   {\Bigg(\frac{1}{N_K^{\frac{1}{2(d-1)}}}+\frac{\rho_{0}}{8(d-1)} \frac{\varepsilon}{K} \Bigg)}^{-2(d-1)}. \label{eqsecondmomentsecondexp}
        \end{align} Combining Eq.~\eqref{eqsecondmomentfirstexp} and Eq.~\eqref{eqsecondmomentsecondexp}, we arrive at  \[\limsup_{K \rightarrow \infty} \mathbb{E} \bigg[  \sup_{t \in [\varepsilon,T]} \langle  \mu_{\otimes}^{K}(t), \|\cdot \|_{2}^{2}\rangle   \bigg] \leq \amax^2T^2 + \Big(\frac{8(d-1)}{\rho_{0} \varepsilon} \Big)^{2(d-1)} < \infty. \qedhere  \]    
	\end{proof}
	\subsubsection{On the proofs of Lemma \ref{lemmap3}} 
We now turn to establishing the continuity of the accumulation points of the sequence $\{\omu\}_{K\geq 1}$.
	\begin{proof}[Proof of Lemma \ref{lemmap3}]
       We start with the critical sampling case. Let $\mu_{\otimes}^{\infty}$ be an accumulation point of the sequence $\left\{ \mu_{\otimes}^{K} \right\}_{K \geq 1}$ in $D([0,T], (\mathcal{M}_{f} (\mathbb{R}_{+}^{d^{2}}), w))$. By a slight abuse of notation we denote by $\left\{ \mu_{\otimes}^{K} \right\}_{K \geq 1}$ the subsequence converging to  $\mu_{\otimes}^{\infty} $. Recall that \begin{itemize}
		    \item a migration from colony $j$ to colony $i$ moves a single point measure with a mass of $1/K$ from $\mu_j^{K}$ to $\mu_i^K$
            \item a coalescence removes two point measures each with a mass of $1/K$ and adds another point measure, also with a mass of $1/K$.
		\end{itemize} Using this and that $\langle \mu_{i}^{K}(t), 1 \rangle \leq b$, it follows that  
     \begin{align*}
	\sup_{t \in [0,T]} &\sup_{f \in L^{\infty}(\mathbb{R}_{+}^{d^2}), \|f\|_{\infty}\leq 1} | \langle \mu_{\otimes}^{K} (t), f \rangle  - \langle \mu_{\otimes}^{K} (t-), f \rangle | \\
    &\leq \sup_{t \in [0,T]} \frac{1}{K} \bigg(4 \prod_{\ell \in [d] \setminus \{i\}} \langle \mu_{\ell}^{K}(t), 1 \rangle + \prod_{\ell \in [d] \setminus \{j\}} \langle \mu_{\ell}^{K}(t), 1 \rangle + \prod_{\ell \in [d] \setminus \{i,j\}} \frac{1}{K} \langle \mu_{\ell}^{K}(t), 1 \rangle   \bigg) \\    
    &\leq   \frac{5b^{d-1}}{K} + \frac{b^{d-2}}{K^2}   \underset{K \rightarrow \infty}{\longrightarrow} 0 \text{.} 
    \end{align*}  It follows that  $\mu_{\otimes}^{\infty} \in C([0,T], (\mathcal{M}_{f} (\mathbb{R}_{+}^{d^{2}}), w))$.\\
    Let us now consider the large sampling case. Similarly as before, let $\mu_{\otimes}^{\infty}$ be an accumulation point of the sequence $\left\{ \mu_{\otimes}^{K} \right\}_{K \geq 1}$ in $D(\varepsilon,T], (\mathcal{M}_{f} (\mathbb{R}_{+}^{d^{2}}), w))$. By a slight abuse of notation we denote by $\left\{ \mu_{\otimes}^{K} \right\}_{K \geq 1}$ the subsequence converging to  $\mu_{\otimes}^{\infty} $. This time we obtain
  \begin{align*}
	\sup_{t \in [\varepsilon,T]} &\sup_{f \in L^{\infty}(\mathbb{R}_{+}^{d^2}), \|f\|_{\infty}\leq 1} | \langle \mu_{\otimes}^{K} (t), f \rangle  - \langle \mu_{\otimes}^{K} (t-), f \rangle | \\
     \leq& \sup_{t \in [\varepsilon,T]} \frac{1}{K} \bigg(4 \prod_{\ell \in [d] \setminus \{i\}} \langle \mu_{\ell}^{K}(t), 1 \rangle + \prod_{\ell \in [d] \setminus \{j\}} \langle \mu_{\ell}^{K}(t), 1 \rangle + \prod_{\ell \in [d] \setminus \{i,j\}} \frac{1}{K} \langle \mu_{\ell}^{K}(t), 1 \rangle   \bigg) \\     
    &\leq  \frac{5}{K}  \left(\frac{L^{K}(\frac{\varepsilon}{K})}{K}\right)^{d-1} +\frac{1}{K^2}  \left(\frac{L^{K}(\frac{\varepsilon}{K})}{K}\right)^{d-2}   \text{,} 
    \end{align*}  
and the result follows using Lemma \ref{lemdelayedkingman} and that, for the Kingman coalescent $\tilde{L}^{\infty}$ started with infinitely many lines,  $t\tilde{L}_t^{\infty}$ converges almost surely as $t\to 0$ to a positive constant (see e.g. \cite[Thm.~1]{berestycki2010lambda} or \cite[Chap.~2.1.2, Eq.~(24)]{berestycki2009recent}).
	\end{proof}

\subsection{Characterization of the accumulation points}\label{ss:charac}
	We conclude with the main result of this section, characterizing the accumulation points of the sequence of vectors of measures $\{\boldsymbol{\mu}^{K}\}_{K \geq 1}$:
	\begin{proposition}\label{propaccum}
		Let $t_0=0$ in the critical sampling case and $t_0=\varepsilon>0$ in the large sampling case. Consider an accumulation point $\boldsymbol{\mu}^{\infty}$ of $\{\boldsymbol{\mu}^{K}\}_{K \geq 1}$. For every $\bp\coloneqq (p_{i})_{i \in [d]}$ with $p_{i} \in D([t_0,T], (\mathcal{M}_{f} (\mathbb{R}_{+}^{d}), w))$, we define \begin{align*}
			\varphi_{t, {i},\bbf} (\bp)&\coloneqq  H^{F_{i},\bbf} (\bp(t))  - H^{F_{i},\bbf} (\bp(t_0)) - \int_{t_0}^{t} \bar{A} H^{ F_{i},\bbf} (\bp(s)) \mathrm{d}s 
		\end{align*} 
         where $F_{i}(\bx)=x_{i}$ is the projection to the $i$-th coordinate and $\bbf= (f, \dots, f)$, for $f \in C_{b}(\mathbb{R}_{+}^{d})$. 
		Then  \[\varphi_{t, {i},\bbf}(\boldsymbol{\mu}^{\infty})=0 \text{, a.s.} \] for every $t \in [t_0,T]$ and any $i \in [d]$.
	\end{proposition}
	
	\begin{proof}
			We prove the result for the large sampling case. By a slight abuse of notation we denote by $\{\boldsymbol{\mu}^{K} \}_{K \geq 1} $a subsequence converging to $\boldsymbol{\mu}^{\infty}$. We then observe from Proposition~\ref{propconvergencegenerators} and Lemma~\ref{lemmaquadraticvariation} \begin{align*}
			\mathbb{E} [|\varphi_{t, {i},\bbf}( \boldsymbol{\mu}^{K})|]
			&\leq  \mathbb{E}[|M_{t}^{K,\{i\},\bbf}-M_{0}^{K,\{i\},\bbf}|] + \mathbb{E} \left[ \int_{\varepsilon}^{t} | (A^{K}-\bar{A}) H^{F_{i}, \bbf} (\boldsymbol{\mu}^{K}(s)) | \mathrm{d}s \right]  
			\end{align*} The first expectation goes to $0$ with $K \to \infty$ according to Lemma~\ref{AR2}. For the second expectation, we find by Proposition~\ref{propconvergencegenerators} \begin{align*}
		    \mathbb{E} \left[ \int_{\varepsilon}^{t} | (A^{K}-\bar{A}) H^{F_{i}, \bbf} (\boldsymbol{\mu}^{K}(s)) | \mathrm{d}s \right] &\leq \mathbb{E} \bigg[ \int_{\varepsilon}^{t}  \frac{C(F_{i},\bbf)}{K}\bigg(\sum_{i\in[d]}  \langle  \mu_{i}^{K}(t), 1 \rangle  +  \langle  \mu_{i}^{K}(t), 1 \rangle^2 \bigg)      \mathrm{d}s \bigg] \\
            &\leq  \frac{C(F_{i},\bbf) (t- \varepsilon)}{K} \mathbb{E} \bigg[  \frac{L^{K}(\frac{\varepsilon}{K})}{K} + \left( \frac{L^{K}(\frac{\varepsilon}{K})}{K}  \right)^{2} \bigg] \underset{K \to \infty}{\longrightarrow} 0,
		\end{align*} due to Lemma~\ref{lemdelayedkingman} and Lemma~\ref{moki}.

        On the other hand, since any accumulation point $\boldsymbol{\mu}^{\infty}$ must be in $C([\varepsilon,T], (\mathcal{M}_{f} (\mathbb{R}_{+}^{d})^{d}, w))$, because $\mu_{\otimes}^{\infty} \in C([\varepsilon,T], (\mathcal{M}_{f} (\mathbb{R}_{+}^{d^{2}}), w))$  and $\mu_{\otimes}^{\infty} \mapsto \boldsymbol{\mu}^{\infty}$ is continuous (and since for every continuous bounded $g$ the process $(\langle p^{K}(t), g \rangle )_{t \geq 0}$ converges to $(\langle p^{\infty}(t), g\rangle )_{t \geq 0}$ in the uniform norm on every finite interval if $(p^{K}(t))_{t \geq 0}$ converges to a continuous $(p^{\infty}(t))_{t \geq 0} $) we also have \[ \varphi_{t, {i},\bbf}(\boldsymbol{\mu}^{K}) \Rightarrow \varphi_{t, {i},\bbf}(\boldsymbol{\mu}^{\infty}) \text{.} \] Therefore we find \[
		\mathbb{E} [|\varphi_{t, {i},\bbf}(\boldsymbol{\mu}^{\infty})|  ] = \lim_{h \rightarrow \infty} \mathbb{E} [|\varphi_{t, {i},\bbf}( \boldsymbol{\mu}^{K})|] = 0  \] by the uniform integrability of  $\{|\varphi_{t, {i},\bbf}( \boldsymbol{\mu}^{K})|\}_{K \geq 1}$. Therefore, $\varphi_{t, {i},\bbf}(\boldsymbol{\mu}^{\infty})=0$ a.s for every $t\in [\varepsilon, T]$ and any $i \in [d]$.
	\end{proof}

\section{Proofs of Theorems ~\ref{theocrit} and~\ref{theolarge}.}\label{secconvergenceproofs}
In this section, we provide the proofs of our convergence results. We begin with the proof of Theorem~\ref{theocrit}, which follows directly from the results obtained in the previous sections. The proof of Theorem~\ref{theolarge}, on the other hand, requires additional work due to the degenerate initial condition.
\subsection{The critical sampling case}
As anticipated, we already have all the necessary ingredients to prove the main result in the critical sampling case.
	
\begin{proof}[Proof of Theorem~\ref{theocrit}]
Proposition~\ref{proptightness} ensures that the sequence $\{\mu^{K}\}_{K \ge 1}$ is tight, while Proposition~\ref{propaccum} identifies all of its accumulation points as weak solutions of Eq.~\eqref{SE2}. Together, these results already guarantee the existence of weak solutions to Eq.~\eqref{SE2}. In addition, under Assumption~\ref{assu1} we have
\begin{equation}\label{icrholds}
    \mu^K_i(0,\{\bn\})=\frac{L^K_i(0)}{K}\delta_{e_i,\bn}
    \xrightarrow[K\to\infty]{} c\beta_i\delta_{e_i,\bn},
\end{equation}
which establishes convergence of the initial condition. The desired convergence thus follows from the uniqueness of solutions to Eq.~\eqref{SE2} with initial condition \eqref{ICa2} (see Prop.~\ref{udiscr} in Appendix~\ref{secuniquenessd>1}).
    \end{proof} 

\subsection{The large sampling case}\label{secinitalcondition}	
As mentioned at the beginning of this section, the proof of Theorem~\ref{theolarge} requires an additional ingredient due to the degenerate initial condition, which comes from the fact that
\[
\langle \mu_i^K(0),1\rangle = \frac{L_i^K(0)}{K} \xrightarrow[K\to\infty]{} \infty.
\]
The next result allows us to circumvent this problem.
\begin{proposition} \label{propconvergenceprobability}
		We have, for $\boldsymbol{\lambda} = (\lambda_{i})_{i \in [d]} > \boldsymbol{0}$,
		\[ \lim_{\varepsilon \downarrow 0} \lim_{K \to \infty} \langle \mu_{i}^{K}(\varepsilon), 1-e^{- \langle   \boldsymbol{\lambda}, \cdot \,\rangle} \rangle = \lambda_{i} \beta_{i}, \textrm{ in probability.}   \] 
	\end{proposition}
    Before we dive into the proof of this result, let us use it to prove Theorem~\ref{theolarge}.
    \begin{proof}[Proof of Theorem~\ref{theolarge}]
Let $\left\{ \bmu^{K_n}   \right\}_{n \geq 1}$ be an arbitrary subsequence of $\left\{ \bmu^{K}   \right\}_{K \geq 1}$. Due to Propositions~\ref{proptightness} and \ref{propaccum}, for each $\varepsilon \in (0,T)$, one can extract a subsequence  $
\{ \bmu^{K_n^{'}} \}_{n \geq 1}$ of $\left\{ \bmu^{K_n}   \right\}_{n \geq 1}$ converging in distribution to a measure-valued process $(\boldsymbol{u}^{\varepsilon}(t,\cdot))_{t\in[\varepsilon,T]}$ evolving according to \eqref{SE1}. Then a standard diagonalization argument yields the existence of a further subsequence $\{ \bmu^{K_n{''}}\}_{n \geq 1}$ of $\{ \bmu^{K_n}\}_{n \geq 1}$ and a measure-valued process $(\boldsymbol{u}(t,\cdot))_{t\in(0,T]}$ such that
\begin{itemize}
\item $\boldsymbol{u}$ evolves according to Eq.~\eqref{SE1},
\item for any $\varepsilon\in(0,T)$, $\boldsymbol{u}|_{[\varepsilon,T]}=\boldsymbol{u}^{\varepsilon}$,
\item for each $\varepsilon \in (0,T)$, 
        $\bmu^{K_n^{''}} \Rightarrow\boldsymbol{u}$ in $[\varepsilon,T]$.
        \end{itemize}
We now claim that $\boldsymbol{u}$  for all $i \in [d]$,
		\[  \lim_{t \downarrow 0} \langle u_{i}(t), 1-e^{- \langle \boldsymbol{\lambda}, \cdot\rangle}  \rangle= \lambda_{i} \beta_{i} \text{, a.s.} \label{InC2} \] 
If the claim holds, we can first conclude that $\boldsymbol{u}$ is a weak solution to \eqref{SE1} satisfying the initial condition \eqref{ICa1}. By Theorem~\ref{thm:stoch-represenatation-large}, the function $\boldsymbol{u}$ is uniquely determined (and in fact becomes a deterministic measure-valued function), and therefore does not depend on the choice of the initial subsequence $\left\{ \bmu^{K_n} \right\}_{n \geq 1}$. The result then follows.
\smallskip

Let us now prove the claim. To simplify the notation, we will denote from now on the subsequence $\{ \mu^{K_n{''}}\}_{n \geq 1}$ as the original sequence $\{ \mu^{K} \}_{K \geq 1}$. By construction, we have for any $\varepsilon>0$

	\[ \langle \mu^{K}_{i}(\varepsilon), 1-e^{- \langle   \boldsymbol{\lambda}, \cdot \rangle} \rangle \Rightarrow \langle u^{}_{i}(\varepsilon), 1-e^{- \langle \boldsymbol{\lambda}, \cdot \rangle} \rangle.  \] 
    
    Therefore, it follows that 
    \[ \liminf_{K \rightarrow \infty} \mathbb{P}(  \langle \mu^{K}_{i}(\varepsilon), 1-e^{- \langle \boldsymbol{\lambda}, \cdot \rangle}  \rangle \in G ) \geq \mathbb{P}(\langle  u_{i} (\varepsilon), 1-e^{- \langle \boldsymbol{\lambda}, \cdot \rangle} \rangle \in G ),  \] for every open set $G$. 
    In particular, for any $\rho >0$: \[ \mathbb{P}( | \langle \boldsymbol{u}_{i}(\varepsilon), 1-e^{- \langle \boldsymbol{\lambda}, \cdot \rangle}  \rangle - \lambda_{i}\beta_i| > \rho    ) \leq \liminf_{K \rightarrow \infty} \mathbb{P}(| \langle  \mu^{K}_{i}(\varepsilon), 1-e^{- \langle \boldsymbol{\lambda}, \cdot \rangle} \rangle - \lambda_{i}\beta_{i}|> \rho     ) .   \]
		
		Taking $\limsup$ with $\varepsilon \downarrow 0$ on both sides, we find by Proposition~\ref{propconvergenceprobability} that \begin{align*} \limsup_{\varepsilon \downarrow 0}  \mathbb{P}( |\langle u_{i}(\varepsilon), 1-e^{- \langle \boldsymbol{\lambda}, \cdot \rangle} \rangle - \lambda_{i}\beta_{i}| > \rho ) 
			&\leq \limsup_{\varepsilon \downarrow 0} \liminf_{K \rightarrow \infty} \mathbb{P}( |\langle \mu^{K}_{i}(\varepsilon), 1-e^{- \langle \boldsymbol{\lambda}, \cdot \rangle}  \rangle - \lambda_{i}\beta_{i}|  > \rho     ) =  0.
		\end{align*}		
		Hence $ v_i(\varepsilon,\bl)\coloneqq \langle  u_{i}(\varepsilon), 1-e^{- \langle \boldsymbol{\lambda}, \cdot \rangle} \rangle$ converges in probability to  $\lambda_{i}\beta_{i}$ as $\varepsilon\to 0$. We will now show that, for each $M\in\Nb$, the previous limit exists almost surely for all $\bl\geq 0$ (component-wise) with $\lVert\bl\rVert_\infty\leq M$, which combined with the convergence in probability would yield the desired result. Fix now $M\in\Nb$ and $\bl\geq 0$ with  $\lVert\bl\rVert_\infty\leq M$.
		
Let $\varepsilon_0>0$ be arbitrary for the moment; we will choose it appropriately later. Define for $0<\varepsilon<\varepsilon_0$
		 $$\theta_i^{\varepsilon}(t,\boldsymbol{\lambda})=\beta_i^{-1} \langle  u_{i}(\varepsilon-t), 1-e^{- \langle \boldsymbol{\lambda}, \cdot \rangle} \rangle,$$
and note that $\theta_i^{\varepsilon}(0,\boldsymbol{\lambda})=\beta_i^{-1}v_i(\varepsilon,\bl)$. 
	Since $\boldsymbol{u}$ solves \eqref{SE1}, one can deduce that $(\theta_i^{\varepsilon})_{i\in[d]}$ solves the multi-dimensional ODE
	\begin{equation*}
	 \partial_t \theta_i^{\varepsilon} \ =  \ \frac{1}{2} (\alpha_i \beta_i) (\theta_i^{\varepsilon})^2-\sum_{j \in [d] \setminus \{i\}} \bigg(w_{j,i} \frac{\beta_j}{\beta_i}\theta_j^{\varepsilon} - w_{i,j} \theta_i^{\varepsilon}\bigg) ,\  \mbox{ and $\theta_i^{\varepsilon}(0,\boldsymbol{\lambda}) =\beta_i^{-1} v_i(\varepsilon,\boldsymbol{\lambda}).$} 
     \end{equation*}
Note that $v_i(\varepsilon,\boldsymbol{\lambda})\leq v_i(\varepsilon,M\boldsymbol{1})$, where $\boldsymbol{1}$ is the $d$-dimensional vector with all coordinates equal to $1$. From the convergence in probability of $v_i(\varepsilon,M\boldsymbol{1})$ to $M\beta_i$, we infer that the event
$$A_M(\varepsilon_0)\coloneqq \bigcup_{\varepsilon<\varepsilon_0}\left\{v_i(\varepsilon,M\boldsymbol{1})\leq M\beta_i+\frac{1}{2}\right\} $$
has probability $1$. 

On the event $A_M(\varepsilon_0)$, there is $\varepsilon_*\in(0,\varepsilon_0)$ such that $v_i(\varepsilon_*,\boldsymbol{\lambda})\leq v_i(\varepsilon_*,M\boldsymbol{1})\leq M\beta_i+1/2$. Since, in addition, $\theta_j^{\varepsilon_*}\geq 0$ for all $j\in[d]$, it follows that $\theta_i^{\varepsilon_*}\leq \phi_i$, where $\phi_i$ solves the ODE
	\begin{equation*}
	 \partial_t \phi_i \ =  \ \frac{1}{2} (\alpha_i \beta_i) \phi_i^2 +w_i \phi_i,\  \mbox{ and $\phi_i(0) =M+\frac{1}{2\beta_i},$} 
     \end{equation*}
where $w_i=\sum_{j\in [d]\setminus \{i\} }w_{i,j}$.	 A simple explicit calculation shows that $\phi_i$ blows up at time
$$t_\infty\coloneqq\frac{1}{w_i}\ln\left(1+\frac{2w_i}{\alpha_i(\beta_i M+\frac12)}\right)>0,$$
and hence $\theta_i^{\varepsilon_*}$ does not blow up before $t_\infty$.   
Setting $\varepsilon_0\coloneqq t_\infty$, we conclude that, on $A_M(\varepsilon_0)$, there is $\varepsilon_*\in(0,\varepsilon_0)$ such that
$\theta_i^{\varepsilon_*}(t,\bl)$ admits a limit as $t\to\varepsilon_*$, which translates into the desired convergence of  $\langle  u_{i}(\varepsilon), 1-e^{- \langle \boldsymbol{\lambda}, \cdot \rangle} \rangle$  as $\varepsilon\to 0$.      \end{proof}   
    
    The proof of Proposition \ref{propconvergenceprobability} builds on the following intuition (see Fig.~\ref{figmonpoly}). At small times, the coalescence rate is much higher than the migration rate. Therefore, we expect that, while coalescence reduces the number of blocks to $\mathcal{O}(K)$, the impact of migration is of a smaller order and can be neglected. More precisely, we will see that in a typical block at colony $i$ the number of elements of colors different from $i$ is much smaller than the scaling factor $\gamma_K$. \\
	\begin{figure}[h]
		\includegraphics[width=0.75\textwidth]{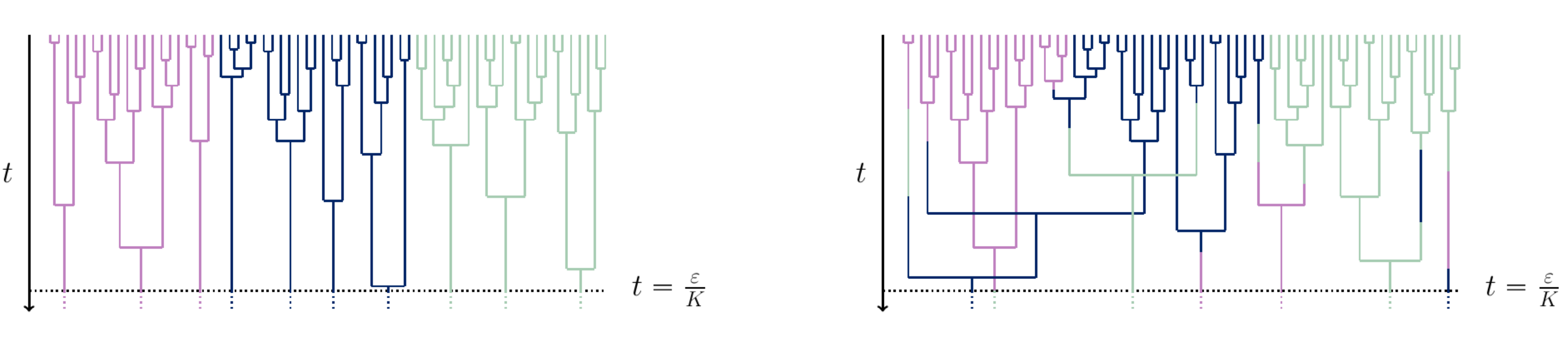}
		\caption{An illustration of the two scenarios. Left: mono-chromatic scenario (configurations at time $\varepsilon/K$: $(5,0,0)$, $(11,0,0)$, $(4,0,0)$ in colony $1$, $(0,5,0)$, $(0,7,0)$, $(0,4,0)$, $(0,7,0)$ in colony $2$, $(0,0,7)$, $(0,0,9)$, $(0,0,7)$ in colony $3$). Right: poly-chromatic scenario (configurations at time $\varepsilon/K$: $(11,0,0)$, $(0,6,0)$, $(0,2,7)$ in colony $1$, $(5,7,0)$, $(0,0,3)$ in colony $2$, $(4,8,0)$, $(0,0, 13)$ in colony $3$).} \label{figmonpoly}
	\end{figure}

	To make this intuition precise, we will introduce two new measures $\bar{\mu}_{i}^{K}$ and $\Delta \mu_{i}^{K}$, the first accounting for the \emph{mono-chromatic part} of the process and the second for the \emph{poly-chromatic part}. To do this, we order the blocks in each colony according to their least element and denote by $B_{i,j,k}^{K} (t)$, $k\in[d]$, the number of elements of type $k$ in the $j$-th block in colony $i$. Note that $\bB_{i,j}^{K}(t)\coloneqq (B_{i,j,k}^{K} (t))_{k \in [d]}$ is then the configuration of the $j$-th block in colony $i$. We then set  \begin{align} \bar{\mu}_{i}^{K}(t) \coloneqq& \frac{1}{K} \sum_{j=1}^{L_{i}^{K}(\frac{t}{K})} \delta_{\frac{\bar{\bB}_{i,j}^{K}(\frac{t}{K})}{\gamma_{K}}}, \qquad\text{where } \bar{B}_{i,j,h}^{K} \Big(\frac{t}{K}\Big) = \begin{cases}
			B_{i,j,i}^{K} (\frac{t}{K}), \text{if } h=i \\
			0, \text{ otherwise,}
		\end{cases}\label{eqmono} \\
		\Delta \mu_{i}^{K}(t) \coloneqq& \frac{1}{K} \sum_{j=1}^{L_{i}^{K}(\frac{t}{K})} \delta_{\frac{\Delta \bB_{i,j}^{K}(\frac{t}{K})}{\gamma_{K}}},\qquad \text{where } \Delta B_{i,j,h}^{K} \Big(\frac{t}{K}\Big) = \begin{cases}
			0, \text{if } h=i \\
			B_{i,j,h}^{K} (\frac{t}{K}), \text{otherwise.}
		\end{cases} \label{eqpoly}
	\end{align} 
    Note that even if $\mu_i^K\neq \bar{\mu}_{i}^{K}+ \Delta \mu_i^K$, we still have for $\boldsymbol{\lambda} \in \mathbb{R}_{+}^{d}$
    $$\mu_{i}^{K}(t) = \frac{1}{K} \sum_{j=1}^{L_{i}^{K}(\frac{t}{K})} \delta_{\frac{(\bar{\bB}_{i,j}^{K}(\frac{t}{K}) + \Delta \bB_{i,j}^{K}(\frac{t}{K}))}{\gamma_{K}}},$$
    and hence
    \[ \langle \mu_{i}^{K}(t), \langle   \boldsymbol{\lambda}, \cdot \,\rangle  \rangle= \langle \bar{\mu}_{i}^{K}(t), \langle   \boldsymbol{\lambda}, \cdot \,\rangle \rangle + \langle \Delta \mu_{i}^{K}(t), \langle   \boldsymbol{\lambda}, \cdot \,\rangle \rangle.\]     
    We will dedicate Subsections~\ref{secmonochrom} and \ref{secpolychrom} respectively to the analysis of the monochromatic and the polychromatic parts in the previous expression. 
	
\subsubsection{Mono-chromatic part}\label{secmonochrom}
The first thing to note here is that the total number of elements of color $i$ counted across all colonies remains constant, as migration simply moves elements from one colony to another. No elements change color, nor are they created or destroyed. So, if we only consider colony $i$, we lose elements of color $i$ due to migration to other colonies, although it is possible that some of them might return (also due to migration).\\
	
The next result, loosely speaking, states that we can neglect the impact of migration at small times $t$.
	
\begin{lem}\label{lemmon}
		Let $\bar{\mu}_{i}^{K}$ be defined as in Eq.~\eqref{eqmono}. Then
		\[ \lim_{\varepsilon \downarrow 0} \lim_{K \to \infty} \mathbb{E} \left[ \langle \bar{\mu}_{i}^{K}(\varepsilon), \langle \boldsymbol{\lambda}, \cdot \rangle \rangle    \right] = \lambda_{i}\beta_{i}.  \]
\end{lem}
\begin{proof}
	 Recall that we start with $L_{i}^{K}(0)$ singletons in colony $i$, and note that  
     \begin{align*}
			\langle \bar{\mu}_{i}^{K}(\varepsilon),  \langle \boldsymbol{\lambda}, \cdot \rangle  \rangle = \frac{1}{K} \sum_{j =1}^{ L_{i}^{K}(\frac{\varepsilon}{K}) } \lambda_{i} \frac{\bar{B}_{i,j,i}^{K}\left(\frac{\varepsilon}{K}\right)}{\gamma_{K}} = \frac{\lambda_{i}}{N_{K}} \sum_{j =1}^{ L_{i}^{K}(\frac{\varepsilon}{K}) } \bar{B}_{i,j,i}^{K}\Big(\frac{\varepsilon}{K}\Big) = \lambda_{i} \frac{L_{i}^{K}(0)- E_{i}^{K}(\frac{\varepsilon}{K})}{N_{K}},
		\end{align*} 
        where $E_{i}^{K}(s)$ denotes the number of elements of color $i$ that are not present in colony $i$ at time $s$. It remains to show \[\lim_{\varepsilon \downarrow 0} \lim_{K \rightarrow \infty} \frac{1}{N_{K}} \mathbb{E}\left[E_{i}^{K}\Big(\frac{\varepsilon}{K}\Big)\right]=0. \]
To do this, we will construct a process $\hat{E}_{i}^{K}$ that upper bounds ${E}_{i}^{K}$; the key ingredients  are 
        	\begin{enumerate}
			\item Migrations of blocks to colony $i$ without any element of color $i$ do not affect $E_{i}^{K}$. Thus, when constructing $\hat{E}_{i}^{K}$ we can remove at time $0$ all singletons with a color different from $i$ and only keep the singletons of color $i$.
			\item When a block in colony $i$ with $k$ elements of color $i$ migrates, it adds $k$ to $E_{i}^K$, but reduces the number of blocks in colony $i$, and therefore the rate of migration out of colony $i$. Thus, when constructing $\hat{E}_{i}^{K}$, at such a transition we will duplicate the block, keeping one copy in colony $i$ and the other contributing to $\hat{E}_{i}^{K}$. This way, when a block migrates back to colony $i$, its elements will have already duplicates in colony $i$, and hence can ignore that migration event. Note that in doing this, we are overcounting the number of elements of color $i$ out fo colony $i$, because the same element may contribute several times to $\hat{E}_{i}^{K}$, but not to $E_i^K$.
		\end{enumerate} \smallskip
        
Having this in mind, we construct our upper-bound process $\hat{E}_{i}^{K}$ as follows (see Fig.~\ref{figatomsbound}). \begin{itemize}
			\item We set $\hat{E}_{i}^{K}(0)=0$ and start a (single-type) Kingman coalescent with rate $\alpha_{i}$ with $L_{i}^{K}(0)$ singletons. 	
            \item At rate $w_{i}K$, with $w_{i}\coloneqq \sum_{j \in [d] \setminus \{i\}} w_{i,j}$, per block  of the Kingman coalescent, we count the number of elements within that block and increase by that amount the value of $\hat{E}_{i}^{K}$. 
		\end{itemize} 
		
		\begin{figure}[h]
		\scalebox{0.8}{
			\includegraphics[width=0.5\textwidth]{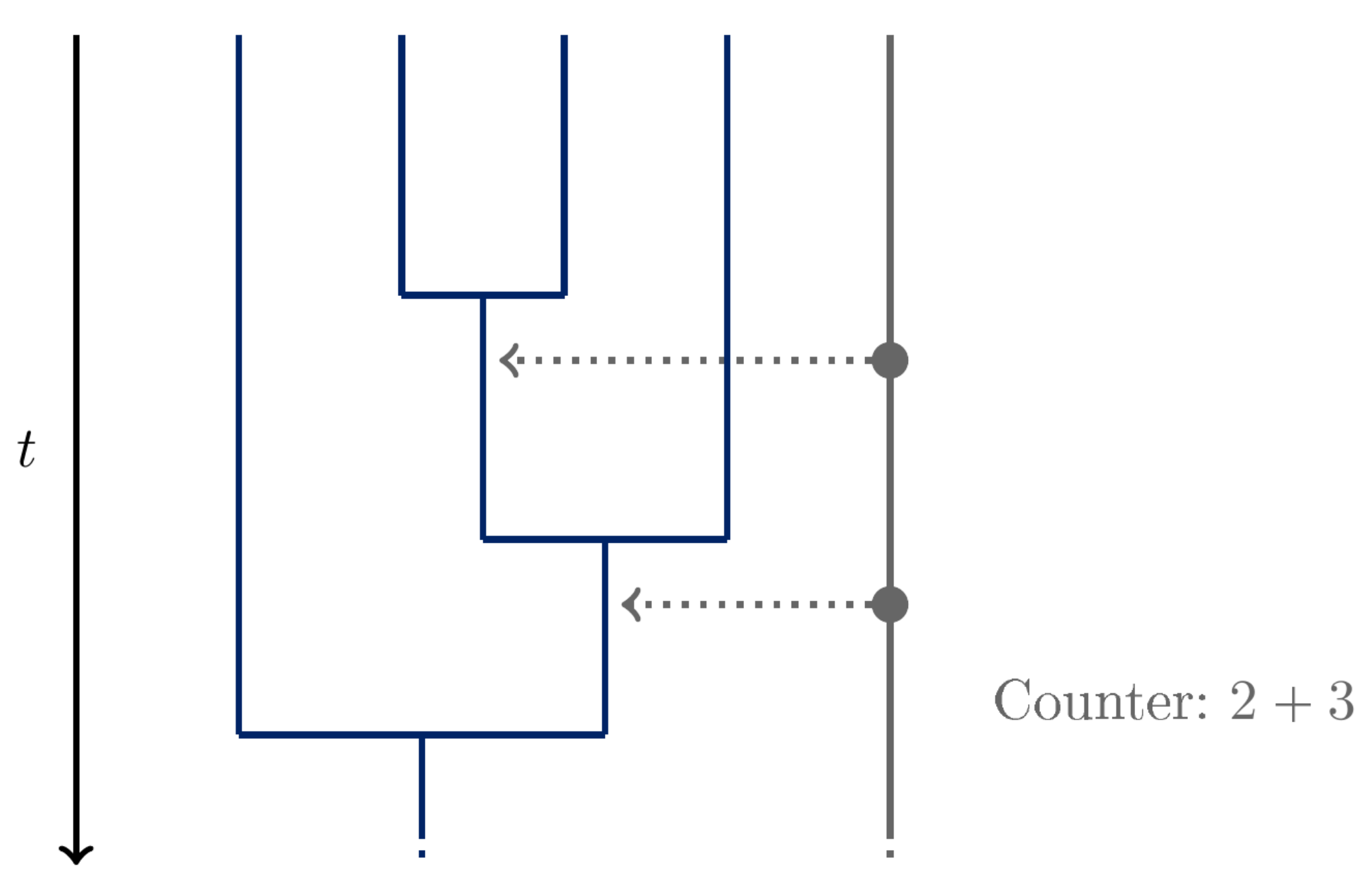}
			}
			\caption{An illustration of $\hat{E}_{i}^{K}$. The circles on the vertical grey line on the right represent the times at which the number of elements within the blocks indicated by the arrows are counted.} \label{figatomsbound}
		\end{figure}
Let $\hat{\mu}_{i}^{K}$ denote the rescaled empirical measure of the single-type Kingman coalescent described above. Then, for functions $F: \mathcal{M}_{f}(\mathbb{R}_{+}) \times \mathbb{N} \rightarrow \mathbb{R}$, the generator of $\big(\hat{\mu}_{i}^{K}(t, \cdot), \hat{E}_{i}^{K}\left(t/K\right) \big)_{t \geq 0}$ is given as \[ \hat{A}^{K}= \hat{A}^{K}_{M}+ \hat{A}^{K}_{C}  \] with \begin{align*}
            \hat{A}^{K}_{M} F(p, m) =& w_{i}K \sum_{n \in \mathbb{N}} p\Big(\Big\{\frac{n}{\gamma_{K}}\Big\} \Big) \left[ F\left(p, m+n \right)-F(p,m) \right], \\
            \hat{A}^{K}_{C} F(p, m) =& \frac{\alpha_{i} K}{2} \sum_{n_{1} \neq n_{2}}  p\Big(\Big\{\frac{n_{1}}{\gamma_{K}}\Big\} \Big) p\Big(\Big\{\frac{n_{2}}{\gamma_{K}}\Big\} \Big) \bigg[ F\Big(p- \frac{\delta_{n_{1}+n_{2}} - \delta_{n_{1}}- \delta_{n_{2}}}{K}, m \Big)-F(p,m) \bigg] \\
            &+ \frac{\alpha_{i} K}{2} \sum_{n \in \mathbb{N}}  p\Big(\Big\{\frac{n}{\gamma_{K}}\Big\} \Big)\Big( p\Big(\Big\{\frac{n}{\gamma_{K}}\Big\} \Big)-1 \Big) \bigg[ F\Big(p- \frac{\delta_{2n} - 2\delta_{n}}{K}, m \Big)-F(p,m) \bigg]. 
        \end{align*} 
        
        Consider first the function $F_1$ defined via $F_{1}(p,m)= \sum_{n \in \mathbb{N}} p\left(\left\{n/\gamma_{K}\right\} \right)n$. Note that $F_1$ is constant in its second component, and that it is invariant under coalescence events. Thus, 
        \begin{equation}\label{F1-exp}
            \mathbb{E}\bigg[ F_{1}\Big(\hat{\mu}_{i}^{K}(t, \cdot), \hat{E}_{i}\Big(\frac{t}{K}\Big)\Big)  \bigg] =  \mathbb{E}\bigg[ F_{1}\Big(\hat{\mu}_{i}^{K}(0, \cdot), 0\Big)  \bigg] = \frac{L_{i}^{K}(0)}{K}.        \end{equation}
        
        Consider now the function $F_2$ defined via $F_{2}(p, m)=m$. Since $F_2$ is constant in its first component, we have $\hat{A}^{K}_{C} F_{2}(p,m)=0$. Moreover, we have 
        \[ \hat{A}^{K}_{M} F_{2}(p,m)= w_{i}K \sum_{n \in \mathbb{N}} p\Big(\Big\{\frac{n}{\gamma_{K}}\Big\} \Big)n = w_{i}K \ F_{1}(p,m).   \] 
        Since $F_{2}(\hat{\mu}_{i}^{K}(0, \cdot), \hat{E}_{i}(0))= 0$, using Dynkin's formula and Eq.~\eqref{F1-exp} yields
        \[ \frac{\mathrm{d}}{\mathrm{d}t} \mathbb{E}\Big[ \hat{E}_{i}\Big( \frac{t}{K} \Big)  \Big] = w_{i}K \  \mathbb{E}\bigg[ F_{1}\Big(\hat{\mu}_{i}^{K}(t, \cdot), \hat{E}_{i}\Big(\frac{t}{K}\Big)\Big)  \bigg] = w_{i} L_{i}^{K}(0),\]
        and therefore, $$\mathbb{E}\Big[ \hat{E}_{i}\Big( \frac{t}{K} \Big)  \Big] = w_{i} L_{i}^{K}(0) t.$$ 
        In particular \[0\leq \limsup_{K \rightarrow \infty} \frac{1}{N_{K}} \mathbb{E}\Big[E_{i}^{K}\Big(\frac{\varepsilon}{K}\Big)\Big] \leq \lim_{K \rightarrow \infty} \frac{1}{N_{K}} \mathbb{E}\Big[\hat{E}_{i}^{K}\Big(\frac{\varepsilon}{K}\Big)\Big]=  w_{i} \beta_{i} \varepsilon.\]
        The result follows taking $\varepsilon\to 0$.
	\end{proof}

\subsubsection{Poly-chromatic part}\label{secpolychrom} The next result tells us that in colony $i$ the number of elements of a different color than $i$ is much smaller than $\gamma_K$. 
	\begin{lemma}\label{lempoly}
		Let $\Delta \mu_{i}^{K}$ be defined as in \eqref{eqpoly}. Then \[ \lim_{\varepsilon \downarrow 0} \lim_{K \to \infty} \mathbb{E} \left[ \langle \Delta\mu_{i}^{K}(\varepsilon),  \langle \boldsymbol{\lambda}, \cdot \rangle \rangle \right] = 0. \] 
	\end{lemma}
	\begin{proof} A straightforward calculation yields \begin{align*}
			\langle \Delta\mu_{i}^{K}(\varepsilon),  \langle \boldsymbol{\lambda}, \cdot \rangle \rangle	 &= \frac{1}{K} \sum_{j =1}^{ L_{i}^{K}(\frac{\varepsilon}{K}) } \sum_{h \in [d]} \lambda_{h} \frac{\Delta B_{i,j,h}^{K}(\frac{\varepsilon}{K})}{\gamma_{K}} \\
            &\leq \frac{\lmax}{N_{K}} \sum_{j =1}^{ L_{i}^{K}(\frac{\varepsilon}{K}) } \sum_{h \in [d]} \Delta B_{i,j,h}^{K}\Big(\frac{\varepsilon}{K}\Big)
            = \frac{\lmax}{N_{K}} I_{i}^{K}\Big(\frac{\varepsilon}{K}\Big) ,
		\end{align*} where $\lmax \coloneqq \max_{i \in [d]} \lambda_{i}$ and $ I_{i}^{K}(t) $ denotes the number of elements, with a color different than $i$, present in colony $i$ at time $t$. 
		It remains to show \[\lim_{\varepsilon \downarrow 0} \lim_{K \rightarrow \infty} \frac{1}{N_{K}} \mathbb{E}\Big[I_{i}^{K}\Big(\frac{\varepsilon}{K}\Big)\Big]=0. \] 
        For this, note first that 
        $$I_{i}^{K}(\varepsilon/K)\leq\sum_{j \in [d]\setminus \{i\}} {E}_{j}^{K}(\varepsilon/K),$$ 
        where ${E}_{j}^{K}$ is defined in the proof of Lemma~\ref{lemmon}. Using this and the upper bound $\hat{E}_{j}^{K}$ we obtained in the proof of Lemma~\ref{lemmon}, we get
        \begin{align*} 0\leq \limsup_{K \rightarrow \infty} \frac{1}{N_{K}} \mathbb{E}\Big[I_{i}^{K}\Big(\frac{\varepsilon}{K}\Big)\Big] &\leq \limsup_{K \rightarrow \infty} \frac{1}{N_{K}} \mathbb{E} \bigg[ \sum_{j \in [d]\setminus \{i\}} \hat{E}_{j}^{K}\Big(\frac{\varepsilon}{K}\Big) \bigg] =  \limsup_{K \rightarrow \infty}  \frac{1}{N_{K}}  \sum_{j \in [d]\setminus \{i\}} w_{j} L_{j}^{K}(0)\varepsilon\\
        &=\varepsilon\sum_{j \in [d]\setminus \{i\}} w_{j} \beta_{j}.
		\end{align*}
        The result follows taking $\varepsilon\to 0$.	\end{proof}
We conclude this section with the proof of Proposition~\ref{propconvergenceprobability}.
	\begin{proof}[Proof of Proposition~\ref{propconvergenceprobability}]
		Due to Markov's inequality we have for all $\rho >0$: \[ \mathbb{P}(| \langle \mu_{i}^{K}(t), 1-e^{- \langle   \boldsymbol{\lambda}, \cdot \rangle} \rangle - \lambda_{i}\beta_{i} | > \rho) \leq \frac{1}{\rho}\mathbb{E}\big[ \langle \mu_{i}^{K}(t), 1-e^{- \langle \boldsymbol{\lambda}, \cdot \rangle}  \rangle - \lambda_{i}\beta_{i}    \big]. \] Therefore, all that remains to do, is to show that \[ \lim_{t \downarrow 0}\lim_{K \to \infty} \mathbb{E}\big[ \langle \mu_{i}^{K}(t), 1-e^{- \langle \boldsymbol{\lambda}, \cdot \rangle}  \rangle   \big]  = \lambda_{i}\beta_{i}.     \]
		To accomplish this, we first write
		\begin{align*}
			\langle \mu_{i}^{K}(\varepsilon), 1-e^{- \langle \boldsymbol{\lambda}, \cdot \rangle} \rangle = \langle \mu_{i}^{K}(\varepsilon),  \langle \boldsymbol{\lambda}, \cdot \rangle \rangle + R_{\boldsymbol{\lambda}}^{K}(\varepsilon).
		\end{align*} 
        with $R_{\boldsymbol{\lambda}}^{K}(\varepsilon)=\langle \mu_{i}^{K}(\varepsilon), 1-\langle \boldsymbol{\lambda}, \cdot \rangle-e^{- \langle \boldsymbol{\lambda}, \cdot \rangle} \rangle$.
       Since we have \[ \langle \mu_{i}^{K}(\varepsilon),  \langle \boldsymbol{\lambda}, \cdot \rangle \rangle = \langle \bar{\mu}_{i}^{K}(\varepsilon),  \langle \boldsymbol{\lambda}, \cdot \rangle \rangle + \langle \Delta \mu_{i}^{K}(\varepsilon),  \langle \boldsymbol{\lambda}, \cdot \rangle \rangle,   \] 
     Lemmas~\ref{lemmon} and \ref{lempoly} imply
       \[ \lim_{\varepsilon \downarrow 0} \lim_{K \to \infty} \mathbb{E}\left[  \langle \mu_{i}^{K}(\varepsilon),  \langle \boldsymbol{\lambda}, \cdot \rangle \rangle    \right] =  \lambda_{i} \beta_{i}.  \] 
       Moreover, since $x-x^2/2 \leq 1-e^{-x} \leq x $, we infer that \[ |R_{\boldsymbol{\lambda}}^{K}(\varepsilon)| \leq \frac{1}{2} \langle \mu_{i}^{K}(\varepsilon) , \langle \boldsymbol{\lambda}, \cdot \rangle^2 \rangle \leq \frac{\lmax^{2}}{2} \langle \mu_{i}^{K}(\varepsilon), \|\cdot\|_{1}^{2} \rangle \leq \frac{\lmax^{2}}{2} \sum_{i \in [d]}   \langle \mu_{i}^{K}(\varepsilon), \|\cdot\|_{1}^{2} \rangle. \] Therefore, using Eq.~\eqref{lemsecondmoment} from Lemma~\ref{lemlemsecondmoment} we obtain \[  \mathbb{E} [ |R_{\boldsymbol{\lambda}}^{K}(\varepsilon)| ] \leq \frac{\lmax^{2} b^2}{2}  \Big(\frac{1}{\gamma_{K}} + \amax \, \varepsilon\Big),\]
       and the result follows letting first $K\to\infty$ and then $\varepsilon \to 0$. 
	\end{proof}


\appendix	
\section{Uniqueness and stochastic representation for the critical sampling}\label{secuniquenessd>1}
		Before diving into the proof of Theorem ~\ref{propuniquenesscrit}, let us prove the following result about the uniqueness of solutions to the discrete coagulation equation.
        \begin{proposition}\label{udiscr}
         The discrete coagulation equation \eqref{SE2} with initial condition \eqref{ICa2} has at most one solution.
        \end{proposition}
        \begin{proof}
       Assume that $\bu$ is a solution of \eqref{SE2} satisfying the initial condition \eqref{ICa2}. Define $\bro \coloneqq \langle \bu, 1 \rangle$. A straightforward calculation shows that $\bro$ solves the system of equations
        \begin{equation}\label{tmeq}
        \partial_t \rho_i = -\frac{\alpha_i}{2} \rho_i^2 + \sum_{j \in [d] \setminus \{i\}} (w_{j,i}\rho_j - w_{i,j}\rho_i),\quad i\in[d],        \end{equation} 
        with initial condition  $\bro(0)= {c} \bb$. According to Picard-Lindelöf theorem, that initial value problem has at most one solution. Since we have one solution, that solution is uniquely determined. \\
        
		Second, let us now fix $\bn_{0} \in \mathbb{N}_{0}^{d}\setminus \{ \boldsymbol{0} \}$. Note that $\bu$ solves the system of non-linear odes \begin{align*} \partial_t u_{i}(t, \bn) = \alpha_{i} \left( \frac{1}{2} \sum_{\bn_{1}+\bn_{2}=\bn} u_{i}(t,\bn_{1}) u_{i}(t,\bn_{2}) - \rho_{i}(t) u_{i}(t,\bn) \right) + \ \sum_{j} (w_{ji} &u_{j}(t,\bn) - w_{ij} u_{i}(t,\bn)), \\
        &\qquad \qquad \quad  i\in[d] \text{, } \bn \in [\bn_{0}], 
        \end{align*}        
		where $\bro$ is the solution of the system of equations \eqref{tmeq} with initial condition $\bro(0)=c\bb$. This is a finite dimensional system, and its solution is again uniquely determined by the initial condition, due to the Picard-Lindelöf theorem. The result follows as $\bn_0$ was arbitrarily chosen.
\end{proof}
	Now we proceed with the proof of Theorem~\ref{propuniquenesscrit} and therefore assume that $\bb=\bxi$ is the equilibrium probability measure, i.e.
	\[	\sum_{j \in [d] \setminus \{i\}} \beta_j w_{j,i} \ = \ \beta_i \sum_{j \in [d] \setminus \{i\}} w_{i,j}. \]

	\begin{proof}[Proof of Theorem ~\ref{propuniquenesscrit}]
	Let $\bu$ be the solution of the discrete coagulation equation \eqref{SE2} under initial condition \eqref{ICa2}. For ${ \boldsymbol{\lambda}} = (\lambda_{i})_{i \in [d]}\in[0,1]^d$, define $\hat{u}_i(t,{\boldsymbol{\lambda}})\coloneqq\sum_{\bn \in \mathbb{N}_{0}^{d} \setminus \{ \boldsymbol{0} \}} u_i(t,\bn ) { \boldsymbol{\lambda}}^\bn$ and	\begin{align*}
		\ v_{i}(t,{\boldsymbol{\lambda}}) &=   \left(1-\frac{\langle u_{i}(t), 1 \rangle}{{c}\beta_i} \right) \ + \ \frac{1}{{c}\beta_i} \hat u_i(t,{\boldsymbol{\lambda}}),
	\end{align*}
	Hence,
	\[{c} \beta_i \partial_t v_{i}(t,\boldsymbol{\lambda}) \ = \ -\partial_t \langle u_{i}(t), 1 \rangle + \sum_{\bn \in \mathbb{N}_{0}^{d} \setminus \{ \boldsymbol{0} \}} \partial_t \hat u_{i}(t, \bn) { \boldsymbol{\lambda}}^\bn. \]
	For the first term we have
	\begin{equation}\label{eqtotalmass}
		\partial_t \langle u_i, 1 \rangle = -\frac{\alpha_i}{2} \langle u_i, 1 \rangle^2 + \sum_{j \in [d] \setminus \{i\}} w_{j,i} \langle u_j, 1 \rangle - w_{i,j} \langle u_i, 1 \rangle. 
	\end{equation}
	For the second term, recalling that $\bu$ solves \eqref{SE2},  we find 
	\begin{align}\label{eqdiscretestochastic2}
		\sum_{\bn\in \mathbb{N}_{0}^{d} \setminus \{ \boldsymbol{0} \}} \partial_t u_{i}(t,\bn) { \boldsymbol{\lambda}}^\bn  = &  \alpha_i\bigg( \frac{1}{2} (\hat u_i)^2 - \langle u_i, 1 \rangle \hat u_i \bigg) + \sum_{j \in [d] \setminus \{i\}} (w_{j,i}  \hat u_j - w_{i,j} \hat u_i) \nonumber \\
		= &   \alpha_i\bigg( \frac{{c^2} \beta_i^2}{2} \Big(v_i - \big(1-\frac{1}{{c}\beta_i}\langle u_i, 1 \rangle \big)\Big)^2 - \langle u_i, 1 \rangle({c}\beta_i v_i - ( {c}\beta_i - |u_i | ))  \bigg) \nonumber 
		\\ & + \sum_{j \in [d] \setminus \{i\}}  w_{j,i}  ({c}\beta_j v_j - ({c}\beta_j - \langle u_j, 1 \rangle) - w_{i,j} ({c} \beta_i v_i  - ({c}\beta_i - \langle u_i, 1 \rangle) .
	\end{align}
Combining the previous identities yields after a tedious, but straightforward calculation  
	\begin{align}\label{geneqdiscr}
		 \partial_t v_{i}(t,{ \boldsymbol{\lambda}})  = & c \alpha_i \beta_i \left( \frac{1}{2} v_i^2 + \frac{1}{2} - v_i\right) \ + \ \sum_{j \in [d] \setminus \{i\}} \Big(  \frac{\beta_j}{\beta_i} v_j w_{j,i} -  v_i w_{i,j} \Big) - \sum_{j \in [d] \setminus \{i\}} \Big( \frac{\beta_j}{\beta_i} w_{j,i} -  w_{i,j} \Big).
	\end{align} 
	\noindent In particular, if we set $$d_i\coloneqq \frac{c \alpha_i \beta_i}{2} - \sum_{j \in [d] \setminus \{i\}}\left( \frac{\beta_j}{\beta_i} w_{j,i}-w_{i,j}\right),$$ we may rewrite \eqref{geneqdiscr} as
	\begin{align}\label{eqateq}
		\partial_t v_{i}(t,{\boldsymbol{\lambda}})  = \  \frac{c \alpha_i \beta_i}{2} v_i^2 + d_i - \bigg( \frac{c \alpha_i \beta_i}{2}+d_i  \bigg) v_i  + \sum_{j \in [d] \setminus \{i\}} \frac{\beta_j}{\beta_i}  w_{j,i}(v_j -v_i).
	\end{align}  
    Recall that by assumption, for each $i\in[d]$, $d_i\geq 0$. Let us now consider a continuous-time multi-type branching process $\bZ(t) = (Z_i(t))_{i\in[d]}$ such that, each particle of type $i\in[d]$
    \begin{itemize}
		\item branches at rate $\frac{{c}  \alpha_{i} \beta_{i}}{2}$,
		\item dies at rate $d_i$,
		\item makes a transition from $i$ to $j$ at rate $\frac{\beta_j}{\beta_i} w_{j,i}$.
	\end{itemize}

	\noindent It is straightforward to check that the function $\boldsymbol{h}(t,{ \boldsymbol{\lambda}})\coloneqq {\mathbb E}_{\be_i}\left( { \boldsymbol{\lambda}}^{\bZ(t)} \right)$ satisfies the system of equations \eqref{eqateq} and the initial condition $\boldsymbol{h}(0,\boldsymbol{\lambda})=\boldsymbol{\lambda}$ (see e.g. \cite[Chap.~V.7, Eq.~(2)]{athreya2004branching}). Thus, by uniqueness of this initial value problem, we infer that
	$$
	v_i(t,{ \boldsymbol{\lambda}}) \ = \ {\mathbb E}_{\be_i}\left( { \boldsymbol{\lambda}}^{\bZ(t)} \right),
	$$
	and the result follows.
\end{proof}

\subsection*{Acknowledgements}

\noindent The authors would like to thank Sebastian Hummel for insightful discussions at an early stage of the project. Fernando Cordero and Sophia-Marie Mellis are funded by the Deutsche Forschungsgemeinschaft (DFG, German Research Foundation) --- Project-ID 317210226 --- SFB 1283.

	\bibliographystyle{abbrvnat}
	\bibliography{Literature}

\end{document}